\documentclass{amsart}

\usepackage{amsthm,amsmath,amsfonts,amssymb}
\usepackage{mathrsfs,latexsym,enumitem}
\usepackage{graphicx,caption,subcaption}
\usepackage{tikz}

\newtheorem{theorem}{Theorem}[section]
\newtheorem{lemma}{Lemma}[section]
\newtheorem{proposition}{Proposition}[section]

\newtheorem{remark}{Remark}[section]

\numberwithin{equation}{section}

\begin{document}
\title{On exact scaling log-Infinitely divisible cascades}

\author{Julien Barral}
\address{LAGA, Institut Galil\'ee Universit\'e Paris 13, Sorbonne Paris Cit\'e, 99 Av. Jean-Baptiste Cl\'ement 93430 Villetaneuse, France}
\email{barral@math.univ-paris13.fr}

\author{Xiong Jin}
\address{Mathematical Institute, University of St Andrews, North Haugh, St Andrews, Fife, KY16 9SS, Scotland}
\email{xj3@st-andrews.ac.uk}
\thanks{The second author is supported by a Royal Society Newton International Fellowship.}

\begin{abstract}
In this paper we extend some classical results valid for canonical multiplicative cascades to exact scaling log-infinitely divisible cascades. We complete previous results on non-degeneracy and moments of positive orders obtained by Barral and Mandelbrot, and Bacry and Muzy: we provide a necessary and sufficient condition for the non-degeneracy of the limit measures of these cascades, as well as for the finiteness of moments of positive orders of their total mass, extending Kahane's result for canonical cascades.  Our main results are analogues to the results by Kahane and Guivarc'h regarding the asymptotic behavior of the right tail of the total mass. They rely on a new observation made about the cones used to define the log-infinitely divisible cascades; this observation provides a ``non-independent'' random difference equation satisfied by the total mass of the measures. The non-independent structure brings new difficulties to study the random difference equation, which we overcome thanks to Goldie's implicit renewal theory. We also discuss the finiteness of moments of negative orders, and some geometric properties of the support. 
\end{abstract}

\maketitle
\section{Introduction}

 This paper studies fine properties of one of the fundamental models of positive random measures illustrating multiplicative chaos theory, namely limits of log-infinitely divisible cascades. 

Multiplicative chaos theory originates mainly from the intermittent turbulence modeling proposed by Mandelbrot in \cite{Ma72}, who introduced  a non completely rigorously mathematically founded construction of measure-valued log-Gaussian multiplicative processes. As its  mathematical treatment was hard to achieve, the model was simplified by Mandelbrot himself, who considered the so-called limit of  canonical  multiplicative cascades in \cite{Ma74c,Ma74a,Ma74b}. The study of  these statistically self-similar measures gave rise to a number of important contributions that we will describe in a while. In the eighties, Kahane founded multiplicative chaos theory in \cite{Ka85,Ka87a,Ka87b}, in particular for Gaussian multiplicative chaos (but also with applications to random coverings), providing the expected mathematical framework  for Mandelbrot's initial construction. Later, fundamental new illustrations of this theory by grid free statistically self-similar measures appeared, namely the compound Poisson cascades  introduced by Barral and Mandelbrot in \cite{BaMa02} and their generalization in the wide class of log-infinitely divisible cascades built by Bacry and Muzy in \cite{BaMu03}; in particular \cite{BaMu03} found a subclass of log-infinitely divisible cascades whose limits possess a remarkable exact scaling property: let $\mu$ be the measure on $\mathbb R_+$ obtained as  the non-degenerate limit  of such a cascade. There exists an integral scale $T>0$ and a L\'evy characteristic exponent $\psi$ such that  for all $\lambda\in (0,1)$, there exists an infinitely divisible random variable $\Omega_\lambda$, such that  $\mathbb{E}(e^{iq\Omega_\lambda})=\lambda^{-\psi(q)}$ for all $q\in\mathbb R$, and  
\begin{equation}\label{exsc}
(\mu([0,\lambda t]))_{0\le t\le T} \overset{\text{law}}{=}  \lambda e^{\Omega_\lambda} (\mu([0,t]))_{0\le t\le T},
\end{equation}
where on the right hand side $(\mu([0,t]))_{0\le t\le T}$ is independent of $e^{\Omega_\lambda} $. Moreover, $\big ((\mu([u,u+t])_{t\ge 0}\big)_{u\ge 0}$ is stationary, and the $\mu$-measure of any two intervals being away from each other by more than $T$ are independent.  

These measures were built on the real line, and higher dimensional versions have been built as well (see \cite{Ch06,RhVa10} for generalisations to the higher dimension). In particular, in dimension 2 and in the log-Gaussian case, they are closely related to the validity of  the so-called KPZ formula and its dual version in Liouville quantum gravity (see \cite{DuSh11} and references therein, as well as  \cite{BaJiRhVa12}). 

The same series of questions which have interested mathematicians for canonical cascades naturally occur for log-infinitely divisible cascades. This paper will deal with some of them, both by sharpening some known results and proving new ones, especially regarding the right tail asymptotic behavior of the law of the total mass of such a measure restricted to compact intervals. Our study will be based on, in an essential way, an alternative construction of the log-infinitely divisible cascades with exact scaling, consisting in making a new choice of ``cones'' used to build them. This new point of view also turns out to  have the advantage to make it possible to build multifractal processes over $\mathbb R_+$ combining stationarity and long range dependence of their  increments  along the multiples of an integral scale $T$, and exact scale invariance properties at scales smaller than $T$ over the intervals $[nT,(n+1)T]$; however we will lose the global stationarity of the increments, the stationarity being reduced to the semi-group $T\cdot \mathbb N$.  

\medskip

Let us come back to the canonical multiplicative cascades and the related fundamental questions. To build such a random measure in dimension 1, one considers  for instance the dyadic tree
\[
T=\bigcup_{j\ge 1}\Big \{M_{u}=\Big (2^{-(j+1)}+\sum_{k=1}^ju_k 2^{-k},2^{-j}\Big)\Big \}_{u\in\{0,1\}^j}
\]
embedded in the upper half-plane $\mathbb H$ (this extends naturally to $m$-adic trees). Then to each point $M_u$ one associates a random variable $W_u$, so that the $W_u$, $u\in\bigcup_{j\ge 1} \{0,1\}^j$, are independent and identically distributed with a positive random variable $W$ of expectation 1, and one defines a sequence of measures on $[0,1]$ as
\begin{eqnarray*}
\mu_j (\mathrm d t)&=&\prod_{k=1}^j W_{u_1\cdots u_k}\cdot { \mathrm dt}\quad\text{if }t\in \Big [ \sum_{k=1}^ju_k 2^{-k}, 2^{-j}+\sum_{k=1}^ju_k 2^{-k}\Big ),
\end{eqnarray*}
a definition which, to be interpreted in the same setting as that used to define the log-infinitely divisible cascades studied in this paper, can be reformulated in 
$$
\mu_j (\mathrm d t)=e^{\Lambda (C_{2^{-j}}(t))}{\, \mathrm dt},
$$
where $C_{2^{-j}}(t)=\{z=x+iy\in \mathbb{H}: -y/2\le x-t<y/2,\  2^{-j}\le y\le 1 \}$ and $\Lambda$ is the random measure on $(\mathbb H,\mathcal B(\mathbb H))$ defined as 
$$
\Lambda (A)=\sum_{u: M_u\in A}\log (W_u). 
$$
Indeed, the compound Poisson cascades mentioned above correspond formally to the replacement of the tree $T$ by the points of a Poisson point process in $\mathbb H$ with an intensity of the form $ay^{-2} \mathrm dx\mathrm dy$ ($a>0$),  the process being independent of the copies of $W$ attached to its points. 

The sequence $(\mu_j)_{j\ge 1}$ is a martingale which converges almost surely weakly to a measure $\mu$ supported on $[0,1]$.  Mandelbrot was especially interested in three related  questions: (1) under which necessary and sufficient conditions is $\mu$ non-degenerate, i.e. $\mathbb P(\mu\neq 0)=1$ ($\{\mu\neq0\}$ is a tail event of probability 0 or 1)? (2) When $\mu$ is non-degenerate, under which necessary and sufficient conditions $\mathbb E(\|\mu\|^q)<\infty$ when $q>1$? (3) When $\mu$ is non-degenerate, what is the Hausdorff dimension of $\mu$? He formulated and partially solved  conjectures about these questions. Then, the two first questions were solved by Kahane and the third one by Peyri\`ere in \cite{KaPe76}: let 
\begin{equation}\label{phi1}
\varphi(q)=\log_2\mathbb{E}(W^q)-(q-1).
\end{equation}
Then $\mu$ is non-degenerate if and only if $\varphi'(1^-)<0$; in this case the convergence of $\|\mu_j\|$ holds in $L^1$ norm, and  for $q>1$ one has $\mathbb E(\|\mu\|^q)<\infty$ if and only if $\varphi(q)<0$; also, the Hausdorff dimension of $\mu$ is $-\varphi'(1^-)$ (Peyri\`ere assumed $\mathbb E(\|\mu\|\log^+\|\mu\|)<\infty$, a condition removed in \cite{Ka87b}).

Answers to questions (1) and (2) exploited finely  the fundamental equation governing the canonical multiplicative cascade  and its limit (especially its exact scaling properties along the dyadic grid), namely the almost sure relation 
\begin{equation}\label{FEC}
Z=2^{-1}(W_0Z(0)+W_1Z(1)),
\end{equation}
where $Z=\|\mu\|$ and $Z(0)$ and $Z(1)$ are the independent copies of $Z$ obtained by making the substitution $W_u:=W_{0u}$ and $W_u:=W_{1u}$ respectively in the construction. Notice that in \eqref{FEC} we also have $(W_0,W_1)$ being independent of $(Z(0),Z(1))$. 

\medskip

Mandelbrot also raised the question of the asymptotic behavior of the right tail of $Z$. Kahane noticed that all the positive moments of $Z$ are finite if and only if $\mathbb P(W\le 2)=1$ and $\mathbb P(W=2)<1/2$ (recall that this is also equivalent to $\varphi(q)<0$ for all $q>1$), and in this case he showed in \cite{KaPe76} that
\begin{equation}\label{kaa}
\lim_{q\to\infty} \frac{\log \mathbb{E}(Z^q)}{q\log q}=\log_2 \mathrm{ess}\,\sup (W)\le 1.
\end{equation}
When there exists a (necessarily unique since $\varphi(1)=0$ and $\varphi$ is convex) solution $\zeta$ to the equation $\varphi(q)=0$ in $(1,\infty)$, Guivarc'h, motivated by a conjecture in \cite{Ma74a}, showed in \cite{Gu90} that when the distribution of $\log (W)$ is non-arithmetic, there exists a constant $0<d<\infty$ such that
\begin{equation}\label{gua}
\lim_{x\to \infty} x^{\zeta}\mathbb P(Z>x)=d.
\end{equation}
The proof is based on the connection of \eqref{FEC} with the theory of random difference equations.  

An almost necessary and sufficient condition for the finiteness of moments of negative orders of $Z$ have been obtained in \cite{Mol96,Liu01}. To derive a NSC, rather than $Z$ it is convenient to consider $\widehat Z=\widehat WZ$ where $\widehat W$ is a copy of $W$ independent of $Z$. Then combining \cite{BeSc09}, if $\mu$ is non-degenerate,  for $q>0$ one has $\mathbb E (\widehat Z^{-q})<\infty$ if and only if $\varphi(-q)<\infty$, i.e. $\mathbb E(W^{-q})<\infty$. 

\medskip

We will consider the previous problems for the limits of log-infinitely divisible cascades, whose formal definition will be given in Section~\ref{def3}, using a series of definitions given in Sections~\ref{def1} and~\ref{def2}. The new point of view we adopt on the construction of such measures with exact scaling properties yields equation \eqref{fek}, a natural and essential analogue to \eqref{FEC}, to which is associated an analogue to the logarithmic generating function $\varphi$. This equation does not  emerge immediately from Bacry and Muzy's point of view which, nevertheless, provides the scale invariance in law for the mass of intervals, a property which now follows directly from our approach. The question of non-degeneracy  was almost completely solved for compound Poisson cascades in \cite{BaMa02}; the same was done for the finiteness of moments of positive orders, a result extended to general infinitely divisible cascades in \cite{BaMu03}. Thanks to equation \eqref{fek}, we can prove rather easily for the limit $\mu$ of log-infinitely divisible cascades formally the same results as  the sharp result of Kahane on non-degeneracy (Theorem~\ref{nd}) and the finiteness of moments of positive orders for the total mass of the limit of canonical multiplicative cascades (Theorem~\ref{pm});  then, these results also hold for the more general family of  log-infinitely divisible cascades built in \cite{BaMu03},  since changing the shape of the cones used in the definition of the cascade only creates a random measure equivalent to that corresponding to the exact scaling, and the behaviors of such measures are comparable (see \cite[Appendix E]{BaMu03}).

Our main results concern the extension of Kahane's result on the asymptotic behavior of $\mathbb E(\|\mu\|^q)$ when all the moments of positive orders are finite (Theorem~\ref{eo}), and  the extension of Guivarc'h's result on the right tail behavior of the distribution of $\|\mu\|$ in case of moments explosion (Theorem~\ref{tail}); for these results we require the exact scaling property, so that \eqref{fek} holds.  The situation turns out to be  much more involved than that in the case of  canonical cascades, due to the correlations associated with \eqref{fek}, which are absent in \eqref{FEC}. We  first exploit the unexpected fact that in Goldie's approach in \cite{Gol91} to the right  tail behavior of solutions of random difference equations, it is possible to relax some independence assumptions. Then we must show that at the critical moment of explosion $\zeta$, although $\mathbb E(\mu([0,1])^\zeta)=\infty$, we have $\mathbb E(\mu([0,1/2])\mu([1/2,1])^{\zeta-1})<\infty$ under suitable (weak) assumptions, which  yields (in the non-arithmetic case)
$$
\lim_{x\to \infty} x^\zeta \mathbb{P}(\mu([0,1])>x)=\frac{2 \mathbb{E}\left(\mu([0,1])^{\zeta-1}\mu([0,1/2])-\mu([0,1/2])^{\zeta}\right)}{\zeta \varphi'(\zeta) \log 2} \in(0,\infty).
$$
The finiteness of $\mathbb E(\mu([0,1/2])\mu([1/2,1])^{\zeta-1})$, which is direct in the case of canonical cascades, is  rather involved here. 

For reader's convenience we will also extend to log-infinitely divisible cascades the result on finiteness of moments of negative orders mentioned in the previous paragraph (Theorem~\ref{nm}), though with some effort it may be deduced from \cite{BaMa02} and \cite{RhVa11}; they provide some information on the left tail behavior of the distribution of $\|\mu\|$.  Finally, thanks to \eqref{fek} we can quickly give  fine information on the geometry of the support of $\mu$ (Theorem~\ref{support}). 

\medskip

To complete these preliminary considerations, it is worth mentioning that the notes \cite{Ma74a,Ma74b} also  questioned the existence, when the limit $\mu$ is degenerate, of a natural normalization of  $\mu_j$ by a positive sequence $A_j$ such that $\mu_j/A_j$ converges, in some sense, to a non trivial limit. This problem was solved only very recently thanks to progress made in the study of freezing transition for logarithmically correlated random energy models \cite{We11} and in the study of branching random walks  in which a generalized version of \eqref{FEC} appears naturally \cite{AiSh11,Ma11}. Under weak assumptions, when $\varphi'(1^-)=0$, $\mu_j$ suitably normalized converges in probability to a positive random measure $\widetilde \mu$ whose total mass $Z$ still satisfies \eqref{FEC}, but is not integrable,  while when  $\varphi'(1^-)>0$, after normalization $\mu_j$ converges in law to the derivative of some stable L\'evy subordinator composed with the indefinite integral of an independent measure of $\widetilde \mu$ kind \cite{BaRhVa12}.  Previously, motivated by questions coming from interacting particle systems, Durrett and Liggett had achieved in \cite{DuLi83}  a deep  study of the positive solutions of the  equation  \eqref{FEC} assuming that the  equality holds in distribution only. Under weak assumptions, up to a positive  multiplicative constant, the general solution take either the form of the total mass of a non-degenerate measure $\mu$ or of $\widetilde\mu$, or it takes the form of  the increment between 0 and 1 of  some stable L\'evy subordinator composed with the indefinite integral of an independent measure of $\mu$ or $\widetilde \mu$ kind. Similar properties are conjectured to hold for log-infinitely divisible cascades, see (\cite{BaJiRhVa12} and \cite{DuRhShVa12}). 

Let us now come to the definitions (Sections~\ref{def1} and~\ref{def2}) required to build log-infinitely divisible cascades (Section~\ref{def3}), and our main results for the limits of such cascades (Section~\ref{results}).

\subsection{Independently scattered random measures}\label{def1}

Let $\psi$ be a characteristic L\'evy exponent given by
\begin{equation}\label{psi}
\psi: q \in\mathbb{R} \mapsto iaq -\frac{1}{2}\sigma^2q^2+  \int_{\mathbb{R}} \bigl(e^{iq x}-1-iq x \mathbf{1}_{|x|\le1}\bigr)\,  \nu(\mathrm{d}x),
\end{equation}
where $a,\sigma\in \mathbb{R}$ and $\nu$ is a L\'evy measure on $\mathbb{R}$ satisfying
\[
\nu(\{0\})=0 \text{ and } \int_\mathbb{R}  1 \wedge |x|^2 \, \nu(\mathrm dx) <\infty.
\]
Let $\mathbb{H}=\mathbb{R}\times i\mathbb{R}_+$ be the upper half plane and let $\lambda$ be a measure on $\mathbb{H}$ defined as
\[
\lambda(\mathrm dx\mathrm dy)= y^{-2} \mathrm dx \mathrm dy.
\]
Let $\Lambda$ be an homogenous independently scattered random measure on $\mathbb{H}$ with $\psi$ as L\'evy exponent and $\lambda$ as intensity (see \cite{RaRo89} for details). In particular, for every Borel set $B\in\mathcal{B}_\lambda=\{B\in\mathcal{B}(\mathbb{H}): \lambda(B)<\infty\}$ and $q\in\mathbb{R}$ we have
\[
\mathbb{E}\left(e^{i q \Lambda(B)}\right)=e^{\psi(q)\lambda(B)},
\]
and for every at most countable family of  disjoint Borel sets  $\{B_i\}\subset \mathcal{B}_\lambda$, the random variables $\{\Lambda(B_i)\}$ are independent and satisfy 
\begin{equation}\label{idm}
\Lambda\Big (\bigcup_{i}B_i\Big )=\sum_{i}\Lambda(B_i)\quad\text{almost surely}.
\end{equation}
Let $I_\nu$ be the interval of those $q\in\mathbb{R}$ such that
$\int_{|x|\ge1} e^{ q x}\, \nu(\mathrm{d}x)<\infty$. Then the
function $\psi$ has a natural extension to $\{z\in\mathbb{C}: -\mathrm{Im} (z) \in I_\nu\}$. In particular  for any $q \in I_\nu$ and
every $ B\in \mathcal{B}_\lambda$ we have
\[
\mathbb{E}\left(e^{q \Lambda(B)}\right)=e^{\psi(-iq)\lambda(B) }.
\]

Assume that at least one of $\sigma$ and $\nu$ is positive, and assume that $I_\nu$ contains the interval $[0,1]$. We adopt the normalization
\begin{equation}\label{a}
a=-\frac{\sigma^2}{2}-\int_\mathbb{R} \bigl(e^{x}-1- x \mathbf{1}_{|x|\le1}\bigr) \, \nu(\mathrm{d}x).
\end{equation}
Then for $B\in\mathcal{B}_\lambda$ we define
\[
Q(B)=e^{\Lambda(B)},
\]
and by \eqref{a} we have
\begin{equation}\label{Q}
\mathbb{E}(Q(B))=1.
\end{equation}
More generally for $q\in I_\nu$ we have
\begin{equation}\label{char}
\mathbb{E}(Q(B)^q)=e^{\psi(-iq)\lambda(B)}.
\end{equation}

\subsection{Cones and areas}\label{def2}

Let $\mathcal{I}=\{[s,t]: s,t\in \mathbb{R}, s<t\}$ be the collection of all nontrivial compact intervals. For $I=[s,t]\in\mathcal{I}$ denote by $|I|$ its length $t-s$.

For $t\in\mathbb{R}$ define the cone 
\[
V(t)=\{z=x+iy\in \mathbb{H}: -y/2< x-t\le y/2 \}=V(0)+t.
\]

For $I\in\mathcal{I}$ define
\[
V(I)=\bigcap_{t\in I} V(t).
\]

For $I\in\mathcal{I}$ and $t\in I$ define
\[
V^I(t)=V(t)\setminus V(I).
\]

For $I, J\in\mathcal{I}$ with $J\subseteq I$ define
\[
V^I(J)=\bigcap_{t\in J} V^I(t)=V(J)\setminus V(I).
\]

\begin{lemma}\label{area}
For $I,J\in \mathcal{I}$ with $J\subseteq I$ we have
\[
\lambda(V^I(J))=\log \frac{|I|}{|J|}.
\]
\end{lemma}

\begin{proof}
A direct calculation.
\end{proof}

\subsection{Log-infinitely divisible cascades}\label{def3}

\begin{figure}[b]
\begin{subfigure}[t]{0.24\textwidth}
\centering
\begin{tikzpicture}[xscale=0.7,yscale=0.7]
\draw (-1,0) -- (2,0);
\draw [fill=gray] (1,2) -- (0.1,0.2) -- (1/3+0.1,0.2) -- (4/3,2); 
\draw (-2/3,2) -- (1/3,0) -- (4/3,2); 
\draw (-1,2) -- (0,0) -- (1,2);
\draw (0,2) -- (1,0) -- (2,2);
\draw [dashed] (-1,0.2) -- (2,0.2);
\node [below] at (1/3,0) {$t$};
\node [left] at (-1,0.1) {$\epsilon$};
\end{tikzpicture}
\caption{$A^I_\epsilon(t)$}
\end{subfigure}
\begin{subfigure}[t]{0.24\textwidth}
\centering
\begin{tikzpicture}[xscale=0.7,yscale=0.7]
\draw (-1,0) -- (2,0);
\draw (-2/3,2) -- (1/3,0) -- (4/3,2); 
\draw [fill=gray] (-1,2) -- (-0.1,0.2) -- (1/3-0.1,0.2) -- (-2/3,2); 
\draw (-1,2) -- (0,0) -- (1,2);
\draw (0,2) -- (1,0) -- (2,2);
\draw [dashed] (-1,0.2) -- (2,0.2);
\node [below] at (1/3,0) {$t$};
\node [left] at (-1,0.1) {$\epsilon$};
\end{tikzpicture}
\caption{$B^I_\epsilon(t)$}
\end{subfigure}
\begin{subfigure}[t]{0.24\textwidth}
\centering
\begin{tikzpicture}[xscale=0.7,yscale=0.7]
\draw (-1,0) -- (2,0);
\draw [fill=gray] (-1,2) -- (-0.1,0.2) -- (0.1,0.2) -- (1/2,1) -- (0,2); 
\draw (-2/3,2) -- (1/3,0) -- (4/3,2); 
\draw (-1,2) -- (0,0) -- (1,2);
\draw (0,2) -- (1,0) -- (2,2);
\draw [dashed] (-1,0.2) -- (2,0.2);
\node [left] at (-1,0.1) {$\epsilon$};
\node [below] at (1/3,0) {$t$};
\end{tikzpicture}
\caption{$C^I_\epsilon$}
\end{subfigure}
\begin{subfigure}[t]{0.24\textwidth}
\centering
\begin{tikzpicture}[xscale=0.7,yscale=0.7]
\draw (-1,0) -- (2,0);
\draw (-2/3,2) -- (1/3,0) -- (4/3,2); 
\draw [fill=gray] (-2/3,2) -- (1/3-0.1,0.2) -- (1/3+0.1,0.2) -- (2/3,2/3) --(0,2); 
\draw [fill=gray] (1,2) -- (1/2,1) -- (2/3,2/3) --(4/3,2); 
\draw (-1,2) -- (0,0) -- (1,2);
\draw (0,2) -- (1,0) -- (2,2);
\draw [dashed] (-1,0.2) -- (2,0.2);
\node [below] at (1/3,0) {$t$};
\node [left] at (-1,0.1) {$\epsilon$};
\end{tikzpicture}
\caption{$V^I_\epsilon(t)$}
\end{subfigure}
\caption{The gray areas for the corresponding sets.}\label{bm}
\end{figure}
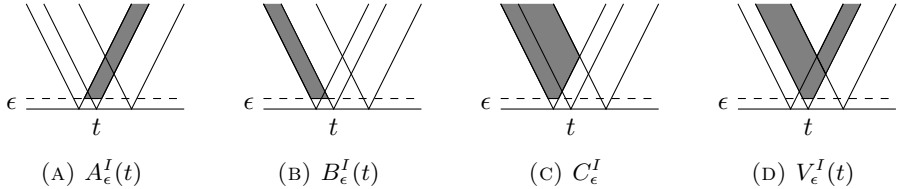

For $\epsilon>0$ denote by
\[
\mathbb{H}_\epsilon=\{z\in\mathbb{H}:  \mathrm{Im}(z)\ge \epsilon\}.
\]
For $I\in\mathcal{I}$, $t\in I$ and $\epsilon>0$ define
\[
V_\epsilon^I(t)=V^I(t)\cap \mathbb{H}_\epsilon.
\]
Clearly we have $V_\epsilon^I(t)\in\mathcal{B}_\lambda$. Moreover, for each $\epsilon>0$ there exists a c\`adl\`ag  modification of $\big (Q(V_\epsilon^I(t))\big )_{t\in I}$. In fact, similar to \cite[Definition 4]{BaMu03}, one can define
\[
\Lambda(V_\epsilon^I(t))=\Lambda(A_\epsilon^I(t))-\Lambda(B_\epsilon^I(t))+\Lambda(C_\epsilon^I),\ t\in I,
\]
where (see Figure \ref{bm})
\begin{eqnarray*}
A_\epsilon^I(t)&=&\{x+iy\in\mathbb{H}:\  y/2\le x \le t+y/2 \}\cap \mathbb{H}_\epsilon,\\
B_\epsilon^I(t)&=&\{x+iy\in\mathbb{H}:\  -y/2\le x \le t-y/2 \}\cap \mathbb{H}_\epsilon,\\
C_\epsilon^I&=&\{x+iy\in\mathbb{H}:\  -y/2\le x \le y/2\wedge(1-y/2) \}\cap \mathbb{H}_\epsilon.
\end{eqnarray*}
It is easy to see that both $\Lambda(A_\epsilon^I(t))$ and $\Lambda(B_\epsilon^I(t))$ are L\'evy processes and $\Lambda(C_\epsilon^I)$ does not depend on $t$, thus $\Lambda(V_\epsilon^I(t))$ has a c\`adl\`ag  modification.

We use this to define $\mu_\epsilon^I$,  the random measure on $I$ given by
\[
\mu_\epsilon^I(\mathrm dx)=\frac{1}{|I|}\cdot Q(V_\epsilon^I(x)) \, \mathrm dx,\ x\in I.
\]
The following lemma is due to Kahane \cite{Ka87a} combined with Doob's regularisation theorem (see \cite[Chapter II.2]{ReYo99} for example).

\begin{lemma}\label{Ka}
Given $I\in\mathcal{I}$, $\{\mu^I_{1/t}\}_{t>0}$ is measure-valued martingale. It possesses a right-continuous modification, which  converges weakly almost surely to a limit $\mu^I$.
\end{lemma}
Throughout, we will work with this right-continuous version of  $\{\mu^I_{1/t}\}_{t>0}$, and its limit $\mu^I$. We give the proof of this lemma with some details, since this point is not made explicit in the context of \cite{BaMu03}. 

\begin{proof}

Let $\Phi$ be a dense countable subset of $C_0(I)$ (the family of nonnegative continuous functions on $I$). Let $f_0$ be the constant mapping equal to 1 over $I$. For $f\in \Phi\cup\{f_0\}$ and $t>0$ define
\[
\mu^I_{1/t}(f)=\int_I f(x) \, \mu^I_{1/t}(\mathrm dx)=\frac{1}{|I|}\int_I f(x) \cdot Q(V_{1/t}^I(x)) \, \mathrm dx
\]
and 
\[
\mathcal F_t=\left(\sigma (\Lambda(V^I_{1/s}(x)): x\in I;\ 0<s\le t)\right)_{t>0}.
\]
Let $\mathcal{N}$ be the class of all $\mathbb{P}$-negligible, $\mathcal{F}_\infty$-measurable sets. Then define $\mathcal{G}_0=\sigma(\mathcal{N})$ and $\mathcal{G}_t=\sigma(\mathcal{F}_t\cup \mathcal{N})$ for $t>0$. Due to the normalisation \eqref{a}, the measurability of $(\omega,t)\mapsto Q(V_\epsilon^I(t))$ and the independence properties associated with $\Lambda$, the family $\{\mu^I_{1/t}(f)\}_{t>0}$ is a positive martingale with respect to the right-continuous complete filtration $(\mathcal{G}_{t})_{t\ge 0}$, with expectation $\mathbb{E}(\mu^I_{1/t})=|I|^{-1} \int_I f(x) \, dx<\infty$. Then from \cite[Chapter II, Theorem 2.5]{ReYo99} one can find a subset $\Omega_0\subset \Omega$ with $\mathbb{P}(\Omega_0)=1$ such that for every $\omega\in\Omega_0$, for each $f\in \Phi\cup\{f_0\}$ and $t\in[0,\infty)$, $\lim_{r\downarrow t; r\in \mathbb{Q}} \mu^I_{1/r}(f)$ exists. Define
\[
\mu^{I,+}_{1/t}(f)=\lim_{r\downarrow t; r\in \mathbb{Q}} \mu^I_{1/r}(f) \text{ if } \omega\in\Omega_0 \text{ and } \mu^{I,+}_{1/t}(f)=0 \text{ if } \omega\not\in\Omega_0.
\]
Then from \cite[Chapter II, Theorem 2.9 and 2.10]{ReYo99} we get that $\mu^{I,+}_{1/t}(f)$ is a c\`adl\`ag modification of $\mu^I_{1/t}(f)$  for each $f\in \Phi\cup\{f_0\}$, thus $\lim_{t\to\infty} \mu^{I,+}_{1/t}(f)$ exists for each $\omega\in \Omega_0$. Now write
\[
\mu^I(f)=\lim_{t\to\infty}  \mu^{I,+}_{1/t}(f) \text{ if } \omega\in\Omega_0 \text{ and } \mu^I(f)=0 \text{ if } \omega\not\in\Omega_0
\]
for each $f\in \Phi$. Since $\Phi$ is a dense subset of $C_0(I)$, one can extend $\mu^{I,+}_{1/t}$ to $C_0(I)$ for each $\omega\in\Omega_0$ by letting
\[
\mu^{I,+}_{1/t}(g)=\lim_{\Phi\ni f\to g} \mu^{I,+}_{1/t}(f), \ g\in C_0(I)
\]
(this limit does exist because for any $f_1,f_2\in\Phi$ and $r\in\mathbb Q$ we have $| \mu^I_{1/r}(f_1)-\mu^I_{1/r}(f_2)|\le  \mu^I_{1/r}(f_0)\|f_1-f_2\|_\infty$). This defines a right-continuous version of $(\mu^{I}_{1/t})_{t>0}$. Then, since the positive linear forms $\mu^{I,+}_{1/t}$ are bounded in norm by  $\mu^{I,+}_{1/t}(f_0)$ and converge over the dense family $\Phi$, they converge. This defines a measure $\mu^I$ as the  weak limit of $\mu^{I,+}_{1/t}$ for each $\omega \in \Omega_0$, hence the conclusion.
\end{proof}

For the weak limit $\mu^I$ we have:

\begin{lemma}\label{copy}
For $I,J\in\mathcal{I}$, $\mu^I\circ f_{I,J}^{-1}$ and $\mu^J$ have the same law, where $f_{I,J}: t\in I \mapsto \inf J+ (t-\inf I)|J|/|I|$.
\end{lemma}
\begin{proof}
Due to the scaling property of $\lambda$ we have that
\[
\left\{Q(V^I_\epsilon(f_{I,J}^{-1}(x)), x\in J\right\} \text{ and } \left\{Q(V^{J}_{\epsilon|J|/|I|}(x), x\in J \right\}
\]
have the same law. This implies that
\[
\left\{\mu_{1/t}^I\circ f_{I,J}^{-1},t>0\right\} \text{ and } \left\{ \mu_{|I|/(|J|t)}^{J},t>0\right\}
\]
have the same law, and so do $\mu^I\circ f_{I,J}^{-1}$ and $\mu^J$.
\end{proof}

Now we come to the scaling property of $\mu^I$. Due to \eqref{idm}, for any fixed compact subinterval $J\subset I$ and $t>0$ we have the decomposition
\begin{equation}\label{scale22}
Q(V_{1/t}^I(x))=Q(V^I(J))\cdot Q(V_{|J|/(|I|t)}^J(x)), \ x\in J,
\end{equation}
hence 
$$
(\mu^I_{1/t})_{|J}=\frac{|J|}{|I|}Q(V^I(J))\cdot \mu^J_{|J|/(|I|t)},
$$
almost surely. Consequently this holds almost surely simultaneously for any at most countable family of such intervals $J$,  but a priori not for all, since $\Lambda$ is not almost surely a signed measure.  This along with Lemma \ref{Ka} and its proof gives simultaneously for all compact intervals $J$ of such a family  the following decomposition
\begin{equation}\label{scale23}
(\mu^I)_{|J}=\frac{|J|}{|I|}Q(V^I(J)) \cdot {\mu}^J
\end{equation}
almost surely, where ${\mu}^I\circ f_{I,J}^{-1}$ has the same law as $\mu^J$, and it is independent of $Q(V^I(J))$ (the fact that $\mu^I$ is continuous assures that the weak limit of $\mu^I_{1/t}$ restricted to $J$ equals $\mu^I$ restricted to $J$; the right-continuous modifications of $(\mu^I_{1/t})_{t>0}$ and the $( \mu^J_{|J|/(|I|t)})_{t>0}$ are built simultaneously, and the convergence of $\mu^I_{1/t}$ implies that of $\mu^J_{|J|/(|I|t)}$). However,  \eqref{scale23} also holds almost surely  simultaneously for all $J\in \mathcal{I}$ with $ J\subset I$ when $\sigma=0$ and the L\'evy measure $\nu$ satisfies $\int 1\land |u|\, \nu(du)<\infty$. Indeed, in this case $\Lambda$ is almost surely a signed measure, which makes it possible to directly write \eqref{scale22} almost surely for all $J\in\mathcal I$ with $J\subset I$ and for all $t>0$  (notice that in this case we easily have the nice property that  almost surely $Q(V_{1/t}^I(x))$ is c\`adl\`ag both in $x$ and $t$). 

We notice that \eqref{scale23} implies \eqref{exsc} (see Section~\ref{connection} for details), but we also have now the  following new equation giving $\|\mu^I\|$ as a weighted sum of its copies: given $k\ge 2$ and $\min I=s_0<\cdots<s_k=\max I$, for $j=0,\cdots,k-1$ write $I_j=[s_j,s_{j+1}]$; provided that $s_1,\cdots,s_{k-1}$ are not atoms of $\mu^I$,  we have almost surely
\begin{equation}\label{fek}
\|\mu^I\|=\sum_{j=0}^{k-1} \frac{|I_j|}{|I|}\cdot Q(V^I(I_j)) \cdot \|{\mu}^{I_j}\|,
\end{equation}
where for each $j$, $\|{\mu}^{I_j}\|$ is independent of $Q(V^I(I_j))$ and has the same law as $\|\mu^I\|$. This equation will be crucial to get our main results. 

\medskip

Another interesting equation is the following. For $I\in \mathcal{I}$ let
\[
I_0=[\min (I),\min(I)+|I|/2] \text{ and } I_1=[\min(I)+|I|/2,\max (I)].
\]
One can also define $I_{00}$ and $I_{01}$ in the same way for $I_0$. Then, provided $I_{00}\cap I_{01}$ is not an atom of $\mu^{I_0}$, we have 
\begin{equation}\label{fek'}
(\mu^I)_{|I_0}=\frac{1}{2}\cdot Q(V^I(I_0))\cdot ((\mu^{I_0})_{|I_{00}}+(\mu^{I_0})_{|I_{01}}),
\end{equation}
where $(\mu^{I_0})_{|I_{00}}\circ f_{I_0,I_{00}}^{-1}$ and $(\mu^{I_0})_{|I_{00}}\circ f_{I_0,I_{01}}^{-1}$ have the same law as $(\mu^I)_{|I_0}$, and they are independent of $\frac{1}{2}Q(V^I(I_0))$. 

\medskip

It remains to prove the following lemma.

\begin{lemma}
Almost surely $\mu^I$ has no atoms.
\end{lemma}

\begin{proof} We can assume that $I=[0,1]$. We start with proving that $1/2$ is not an atom. Let $(f_n)_{n\ge 1}$ be uniformly bounded  sequence in $C_0([0,1])$ which converges pointwise to $\mathbf{1}_{1/2}$, and such that $\mathrm{supp}(f_n)\subset [1/2-\eta_n,1/2+\eta_n]$ with $1/2>\eta_n\downarrow 0$. Then
\begin{eqnarray*}
\mathbb{E}(\mu^I(\{1/2\})) &\le& \liminf_{n\to\infty} \mathbb{E}(\mu^I(f_n))\le  \liminf_{n\to\infty}\liminf_{t\to\infty} \mathbb{E}(\mu^I_{1/t}(f_n))\\
&=& \liminf_{n\to\infty}\int f_n(t)\, \mathrm dt\le \liminf_{n\to \infty} 2\eta_n \|f_n\|_\infty.
\end{eqnarray*}
So $\mathbb{E}(\mu^I(\{1/2\}))=0$.

The fact that $1/2$ is not an atom of $\mu^I$ yields the validity of \eqref{fek'}. Denote by $\widehat \mu=(\mu^I)_{|I_0}$, $\widehat \mu_0=(\mu^{I_0})_{|I_{00}}$, $\widehat \mu_1=(\mu^{I_0})_{|I_{01}}$ and $\widehat W=\frac{1}{2}Q(V^I(I_0))$. From \eqref{fek'} we get
\[
\widehat \mu=\widehat W\cdot(\widehat \mu_0+\widehat \mu_1).
\]
Due to Lemma \ref{copy} we know that whether $\mu^I$ or $\widehat \mu$ having an atom is equivalent. Let $M$ be the maximal $\widehat \mu$-measure of an atom of $\widehat \mu$,  and let $M_j$ be the maximal $\widehat{\mu}_j$-measure of an atom of $\widehat{\mu}_j$ for $j=0,1$. We have $M=\widehat W\max (M_0,M_1)$, where $\widehat W$ is independent of $(M_0,M_1)$, has expectation $1/2$ and $M, M_0, M_1$ have the same law. Thus 
\[
\mathbb{E}(M_0+M_1)/2= \mathbb{E}(M) = \mathbb{E}(\widehat W\max (M_0,M_1))= \mathbb{E}(\max (M_0,M_1))/2.
\]
This implies that, with probability 1, if $M_j>0$ then $M_{1-j}=0$ for $j\in\{0,1\}$. However, $\{M_j>0\}$ is a tail event of probability 0 or 1, thus the previous fact implies that $M_0=M_1=0$ almost surely, hence $\widehat\mu$ has no atoms (here we have adapted to our context the argument of \cite[Lemma A.2]{BeSc09} for canonical cascades). 
\end{proof}

\subsection{Main results}\label{results}

Without loss of generality we may take $I=[0,1]$. For convenience we write $\mu=\mu^{[0,1]}$ and $Z=\|\mu\|$. For $q\in I_\nu$ define
\[
\varphi(q)=\psi(-iq)-(q-1).
\]
Notice that if we set
\[
W=Q(V^{[0,1]}([0,1/2])),
\]
then this function coincides with that of \eqref{phi1} for canonical cascades. 

\medskip

For the non-degeneracy we have

\begin{theorem}\label{nd}
The following assertions are equivalent:
\[
\text{(i) } \mathbb{E}(Z)=1; \text{ (ii) } \mathbb{E}(Z)>0; \text{ (iii) } \varphi'(1^-)<0.
\]
Moreover, in case of non-degeneracy the convergence of $\|\mu_{1/t}^I\|$ to $Z$ holds in $L^1$ norm.
\end{theorem}

For moments of positive orders we have

\begin{theorem}\label{pm}
For $q>1$ one has $0<\mathbb{E}(Z^q)<\infty$ if and only if $q\in I_\nu$ and $\varphi(q)<0$.
\end{theorem}

When $Z$ has finite moments of every positive order we have

\begin{theorem}\label{eo}
(1) The following assertions are equivalent: $(\alpha)$ $0<\mathbb{E}(Z^q)<\infty$ for all $q>1$; $(\beta)$ $\sigma=0$, and $\nu$ is carried by $(-\infty,0]$, $ \int_{-\infty}^0 1\land |x|\, \nu(dx)<\infty$, and
\[
\gamma=\int_{-\infty}^0 \big(1-e^x\big) \, \nu(dx)\le 1.
\]
(2) If $(\beta)$ holds, then
\[
\lim_{q\to \infty} \frac{\log \mathbb{E}(Z^q)}{q\log q}=\gamma.
\]
\end{theorem}

\begin{remark}
Under $(\beta)$ we have for $q\in \mathbb{R}$ and $W=Q(V^{[0,1]}([0,1/2]))$ that
\[
\mathbb{E}(W^{iq})=\exp\left(\Big[iq\gamma+\int_{-\infty}^0 (e^{iqx}-1) \, \nu(dx)\Big]\log 2\right),
\]
which means that $\log W$ is the value at 1 of a L\'evy  process with negative jumps, local bounded variations, and  drift  $\gamma \log 2$, hence $\log_2 \mathrm{ess}\,\sup (W)=\gamma$. This gives in case (2) that
\[
\lim_{q\to \infty} \frac{\log \mathbb{E}(Z^q)}{q\log q}=\log_2 \mathrm{ess}\,\sup (W) \le 1,
\]
which coincides with Kahane's result \eqref{kaa} for canonical cascades.
\end{remark}

In the case where $\mathbb{E}(Z^q)=\infty$ for some $q>1$ we have

\begin{theorem}\label{tail}
Suppose that there exists $\zeta \in I_\nu\cap (1,\infty)$ such that $\varphi(\zeta)=0$; in particular one has $\varphi'(1)<0$. Also suppose that  $\varphi'(\zeta)<\infty$. 

(i) If either $\sigma\neq 0$ or $\nu$ is not of the form $\sum_{n\in \mathbb{Z}} p_n \delta_{nh}$ for some $h>0$, then
\[
\lim_{x\to \infty} x^\zeta \mathbb{P}(Z>x)=d,
\]
where
\[
d=\frac{2 \mathbb{E}\left(\mu([0,1])^{\zeta-1}\mu([0,1/2])-\mu([0,1/2])^{\zeta}\right)}{\zeta \varphi'(\zeta) \log 2} \in(0,\infty).
\]

(ii) If $\sigma=0$ and $\nu$ is of the form $\sum_{n\in \mathbb{Z}} p_n \delta_{nh}$ for some $h>0$, then
\[
0<\liminf_{x\to \infty} x^\zeta \mathbb{P}(Z>x)\le \limsup_{x\to \infty} x^\zeta \mathbb{P}(Z>x) <\infty
\]
\end{theorem}

\begin{remark}\label{r1.2}
From the proof (Remark \ref{tr1.2}) we know that in case (i), when $\zeta=2$,
\[
d=1/\varphi'(2),
\]
which provides us with a family of random difference equations whose solution has a explicit tail probability constant. See \cite{EnSaZi09} for related topics.
\end{remark}

For moments of negative orders we have

\begin{theorem}\label{nm}
Suppose that $\varphi'(1^-)<0$. Then for any $q\in (-\infty,0)$, $\mathbb{E}(Z^q)<\infty$ if and only if $q\in I_\nu$.
\end{theorem}

For the Hausdorff and packing measures of the support of $\mu$ we have

\begin{theorem}\label{support}
Suppose that $\varphi'(1)<0$ and $\varphi''(1)>0$. For $b\in\mathbb{R}$ and $t>0$ let
\[
\psi_b(t)=t^{-\varphi'(1)} e^{b\sqrt{\log^+(1/t)\log^+\log^+\log^+(1/t)}}.
\]
Denote by $\mathcal{H}^{\psi_b}$ and $\mathcal{P}^{\psi_b}$ the Hausdorff and packing measures with respect to the gauge function $\psi_b$ (see \cite{Falconer03} for the definition). Then almost surely the measure $\mu$ is supported by a Borel set $K$ with
\[
\mathcal{H}^{\psi_b}(K)=\left\{\begin{array}{ll}
\infty, & \text{ if } b>\sqrt{2\varphi''(1)},\\
0, & \text{ if } b<\sqrt{2\varphi''(1)},
\end{array}
\right.
\]
and
\[
\mathcal{P}^{\psi_b}(K)=\left\{\begin{array}{ll}
\infty, & \text{ if } b>-\sqrt{2\varphi''(1)},\\
0, & \text{ if } b<-\sqrt{2\varphi''(1)}.
\end{array}
\right.
\]
\end{theorem}

\subsection{Connection with Bacry and Muzy's construction}\label{connection} We may use other shapes for the cone $V$ to define $V(t)=V+t$, for example the one used in \cite{BaMu03} to derive the exact scaling property described in the introduction. The advantage of the present form is that it naturally yields \eqref{scale23} and \eqref{fek}, hence the exact scaling \eqref{exsc}, with $\Omega_\lambda= \Lambda \big (V^{[0,T]}([0,\lambda T])\big )$ if $\mu=\mu^{[0,T]}$. Indeed, for a fixed interval $I$, the measure $\mu^{I}$ has the same law as the restriction to $[0,T]$ of the measure defined from the cone $V^T$ used in \cite{BaMu03} for $T=|I|$, which is drawn on the picture (Figure \ref{Fig2}); this follows from an elementary geometric comparison between the two kinds of cones and the horizontal stationarity of $\Lambda$; otherwise, one can mimic the proof of \cite[Lemma 1]{BaMu03} to get the joint distribution of the $\Lambda$ measures of any finite family of cones  the $\big (V^{[0,T]}_\epsilon(t_1), \ldots,V^{[0,T]}_\epsilon(t_q)\big )$ and find it coincides with the one obtained with the cones $\big (V^{T}_\epsilon(t_1), \ldots,V^{T}_\epsilon(t_q)\big )$.  
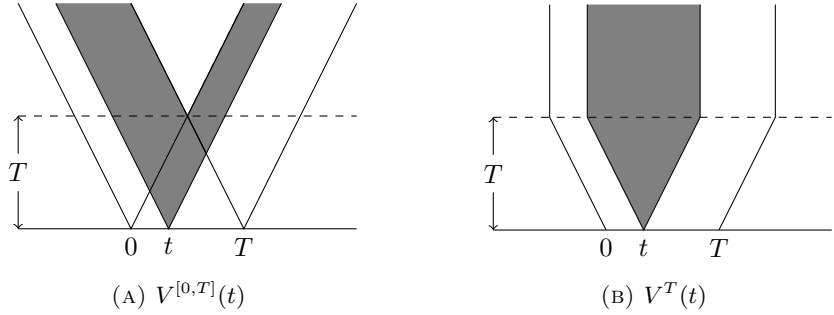
\begin{figure}[ht]
\begin{subfigure}[b]{0.49\textwidth}
\centering
\begin{tikzpicture}[xscale=1.5,yscale=1.5]
\draw (-1,0) -- (2,0);
\draw [fill=gray] (1,2) -- (1/2,1) -- (2/3,2/3) --(4/3,2); 
\draw [fill=gray] (-2/3,2) -- (1/3,0) -- (2/3,2/3) --(0,2); 
\draw (-1,2) -- (0,0) -- (1,2);
\draw (0,2) -- (1,0) -- (2,2);
\draw [dashed] (-1,1) -- (2,1);
\node [below] at (0,0) {$0$};
\node [below] at (1,0) {$T$};
\node [below] at (1/3,0) {$t$};
\draw [->] (-1,2/3) -- (-1,1);
\draw [->] (-1,1/3) -- (-1,0);
\node at (-1,1/2) {$T$};
\end{tikzpicture}
\caption{$V^{[0,T]}(t)$}\label{fig2a}
\end{subfigure}
\begin{subfigure}[b]{0.49\textwidth}
\centering
\begin{tikzpicture}[xscale=1.5,yscale=1.5]
\draw (-1,0) -- (2,0);
\draw [fill=gray] (-1/6,2) -- (-1/6,1) -- (1/3,0) -- (5/6,1) -- (5/6,2); 
\draw (0,0) -- (-1/2,1) -- (-1/2,2);
\draw (1,0) -- (3/2,1) -- (3/2, 2);
\draw [dashed] (-1,1) -- (2,1);
\node [below] at (0,0) {$0$};
\node [below] at (1,0) {$T$};
\node [below] at (1/3,0) {$t$};
\draw [->] (-1,2/3) -- (-1,1);
\draw [->] (-1,1/3) -- (-1,0);
\node at (-1,1/2) {$T$};
\end{tikzpicture}
\caption{$V^T(t)$}\label{fig2b}
\end{subfigure}
\caption{The gray areas for the corresponding sets.}\label{Fig2}
\end{figure}

Using the cones of Figure \ref{fig2b} yields a measure on $\mathbb R_+$, by considering the vague limit of $Q(V_\epsilon^T(t)) \, \mathrm dt$, whose indefinite integral increments are stationary. However, there is no long range dependence between the increments of the indefinite integral of this measure, since two cones have no intersection when associated to points away from each other by at least $T$.  Notice that this measure can also be viewed as the juxtaposition of the limits of $(Q(V_\epsilon^T(t)) \,\mathrm dt)_{|[nT,(n+1)T]}$, $n\in\mathbb N$.  Similarly, consider the measure $\mu$ over $\mathbb R_+$ obtained by juxtaposing the limits of   $(Q(V_\epsilon^{[nT,(n+1)T)}(t))\,\mathrm dt )_{|[nT,(n+1)T]}$. Then, only the process $\mu([nT,(n+1)])_{n\in\mathbb N}$ is stationary, but it has long range dependence: in case of non-degeneracy, if we assume that $\psi(-i2)<\infty$, a calculation shows that 
$$
\mathrm{cov}(\mu([0,T]),\mu([nT,(n+1)])\sim_{n\to\infty} \frac{2\psi(-i2)T^2}{3n},
$$  
so the series $\sum_{n\ge 0}\mathrm{cov}(\mu([0,T]),\mu([nT,(n+1)])$ diverges.

\section{Preliminaries}

Let $\Sigma=\{0,1\}^{\mathbb{N}_+}$ be the dyadic symbolic space. For $\mathbf{i}=i_1i_2\cdots \in\Sigma$ and $n\ge 1$ define $\mathbf{i}|_n=i_1\cdots i_n$. Let $\rho$ be the standard metric on $\Sigma$, that is
\[
\rho(\mathbf{i},\mathbf{j})=2^{-\inf \{n\ge 1: \mathbf{i}|_n=\mathbf{j}|_n\}}, \ \mathbf{i},\mathbf{j}\in\Sigma.
\]
Then $(\Sigma,\rho)$ forms a compact metric space. Denote by $\mathcal{B}$ its Borel $\sigma$-algebra.

For $\mathbf{i}=i_1i_2\cdots \in\Sigma$ define
\[
\pi(\mathbf{i})=\sum_{j=1}^\infty i_j 2^{-j}.
\]
Then $\pi$ is a continuous map from $\Sigma$ to $[0,1]$.

For $n\ge 1$ let $\Sigma_n=\{0,1\}^n$, and use the convention that $\Sigma_0=\{\emptyset\}$.

For $n\ge 0$ and  $i=i_1\cdots i_n \in \Sigma_n$ define
\[
[i]=\{\mathbf{i}\in \Sigma: \mathbf{i}|_n=i\} \text{ and } I_i=\overline{\pi([i])},
\]
with the convention that $i_1\cdots i_0=\emptyset$, $[\emptyset]=\Sigma$ and $I_\emptyset =[0,1]$.

Denote by $\Sigma_*=\cup_{n\ge 0} \Sigma_n$. For $i\in\Sigma_*$ define
\[
W_i=Q(\Lambda(V^I(I_i))) \text{ and } Z_i= \|{\mu}^{I_i}\|.
\]
Then from \eqref{fek} we have for any $n\ge 1$,
\begin{equation}\label{fek2}
2^n Z=\sum_{i\in \Sigma_n} W_i Z_i,
\end{equation}
where $\{W_i,i\in\Sigma_n\}$ have the same law, $\{Z_i,i\in\Sigma_n\}$ have the same law as $Z$ and for each $i\in\Sigma_n$, $W_i$ and $Z_i$ are independent.

\section{Proof of Theorem \ref{nd}}

\subsection{} First we prove (i) $\Leftrightarrow$ (ii) and the $L^1$ convergence. Clearly (i) implies (ii). We suppose that $\mathbb{E}(Z)=c>0$. For any positive finite Borel measure $m$ on $I$ and $t>0$ define
\[
m_t(f)=\frac{1}{|I|}\int_I f(x) \cdot Q(V_{1/t}^I(x)) \, m(\mathrm dx), \ f\in C_0(I).
\]
Following the same argument as in Lemma \ref{Ka}, $m_t$ is a measure-valued right-continuous martingale, thus the Kahane operator $EQ$:
\[
EQ(m)=\mathbb{E}\left(\lim_{t\to \infty} m_t\right)
\]
is well-defined. Denote by $\ell$ the Lebesgue measure restricted to $[0,1]$. Then we have $EQ(\ell)=c\ell$ since $\mathbb{E}(\lim_{t\to\infty}\ell_t(J))=c\ell(J)$ for any compact subinterval $J\subset I$. From \cite{Ka87a} we know that $EQ$ is a projection, so $EQ(EQ(\ell))=EQ(\ell)$. This gives $c=c^2$, hence $c=1$. Consequently, since the limit of the positive martingale $\|\mu^I_{1/t}\|$ with expectation 1 has expectation 1 as well, the convergence also holds in $L^1$ norm.

\subsection{}\label{nda} Now we prove that (ii) implies (iii). From \eqref{fek2} we have that
\begin{equation}\label{2Z}
2Z=W_0Z_0+W_1Z_1.
\end{equation}
Assume that $\mathbb{E}(Z)>0$. For $0<q<1$ the function $x\mapsto x^q$ is sub-additive, hence \eqref{2Z} yields
\begin{equation}\label{suba}
2^q\mathbb{E}(Z^q)\le \mathbb{E}(W_0^qZ_0^q)+\mathbb{E}(W_1^qZ_1^q)=2\mathbb{E}(W_0^q)\mathbb{E}(Z^q).
\end{equation}
Since $\mathbb{E}(Z)>0$ implies $\mathbb{E}(Z^q)>0$, we get from \eqref{suba}, \eqref{char} and Lemma \ref{area} that
\[
2^q \le 2\mathbb{E}(W_0^q)=2e^{\psi(-iq)\log 2}=2^{\psi(-iq)+1}.
\]
This implies $\varphi \le 0$ on interval $[0,1]$, and it follows that $\varphi'(1^-)\le 0$.  To prove $\varphi'(1^-)<0$ we need the following lemma.

\begin{lemma}\label{X01}
Let $X_i=W_iZ_i$ for $i=0,1$.  There exists $\epsilon>0$ such that 
\[
\mathbb{E}(X_0^q \mathbf{1}_{\{X_0\le X_1\}}) \ge \epsilon \mathbb{E}(X_0^q) \ \text{ for } 0\le q \le 1.
\]
\end{lemma}
\begin{proof}
If $\mathbb{E}(X_0^q \mathbf{1}_{\{X_0\le X_0\}})$ is strictly positive for all $q\in [0,1]$, then it is easy to get the conclusion, since both expectations, as functions of $q$, are continuous on $[0,1]$.

Suppose that there exists $q\in(0,1]$ such that $\mathbb{E}(X_0^q \mathbf{1}_{\{X_0\le X_1\}})=0$, then almost surely either $X_0>X_1$ or $0=X_0\le X_1$. Due to the symmetry of $X_0$ and $X_1$ this actually implies that almost surely either $X_0=X_1=0$, or $X_0=0,X_1>0$, or $X_1=0,X_0>0$. This yields
\[
2^q\mathbb{E}(Z^q)=\mathbb{E}(X_0^q)+\mathbb{E}(X_1^q)=2\mathbb{E}(W_0^q)\mathbb{E}(Z^q) \ \text{ for } 0\le q\le 1.
\]
So we have $\psi(-iq)=q-1$ for $q\in[0,1]$. Then from $\frac{\partial^2}{\partial q^2}\psi(-iq)=0$ we get that $\sigma^2=0$ and $\nu\equiv 0$, which is a contradiction to our assumption.
\end{proof}

Now as shown in \cite{KaPe76}, by applying the inequality $(x+y)^q\le x^q +qy^q$ for $x\ge y>0$ and $0<q<1$ we get from \eqref{2Z} and Lemma \ref{X01} that
\[
2^q\mathbb{E}(Z^q) \le 2\mathbb{E}(W_0^q)\mathbb{E}(Z^q) -(1-q)\epsilon\mathbb{E}(W_0^q) \mathbb{E}(Z^q).
\]
This implies
\[
\varphi(q) +\log \left(1-\frac{(1-q)\epsilon}{2}\right)\ge 0 \text{ on } [0,1].
\]
Then it follows that $\varphi'(1^-)-(\epsilon/2\log 2)\le 0$, thus $\varphi'(1^-)<0$.

\subsection{}\label{ndb} Finally we prove that (iii) implies (ii). Assume that $\varphi'(1^-)<0$. For $i\in\Sigma_*$ and $n\ge 1$ define
\[
Y_{n,i}= \mu_{2^{-n}}^I(I_i).
\]
Also denote by $Y_n=\mu_{2^{-n}}^I(I)$. Then for any $m\ge 1$ and $n\ge m+1$ we have
\begin{equation}\label{Zn}
Y_n=\sum_{i\in \Sigma_m} Y_{n,i}.
\end{equation}
We need the following lemma from \cite{KaPe76}.
\begin{lemma}\label{K}
There exists a constant $q_0\in(0,1)$ such that for any $q\in(q_0,1)$ and any finite sequence $x_1,\cdots, x_{k}>0$,
\[
\Big(\sum_{i=1,\cdots, k}x_i\Big)^q \ge \sum_{i=1,\cdots, k} x_i^q -(1-q)\sum_{i\neq j} (x_ix_j)^{q/2}.
\]
\end{lemma}

Applying Lemma \ref{K} to \eqref{Zn} we get for any $q\in(q_0,1)$,
\[
Y_n^q \ge \sum_{i\in \Sigma_m} Y_{n,i}^q-(1-q)\sum_{i\neq j \in \Sigma_m} Y_{n,i}^{q/2}Y_{n,j}^{q/2}.
\]
Taking expectation from both side we get
\begin{equation}\label{bq}
\mathbb{E}(Y_n^q) \ge \sum_{i\in\Sigma_m} \mathbb{E}(Y_{i,n}^q)-(1-q)\sum_{i\neq j \in \Sigma_m} \mathbb{E}(Y_{n,i}^{q/2}Y_{n,j}^{q/2}).
\end{equation}
Let
\begin{eqnarray*}
\mathcal{J}_1&=&\{(i,j)\in \Sigma_m^2: \mathrm{dist}(I_i,I_j)=0\}\\
\mathcal{J}_2&=&\{(i,j)\in \Sigma_m^2: \mathrm{dist}(I_i,I_j)\ge 2^{-m}\}.
\end{eqnarray*}
It is easy to check that $\#\mathcal{J}_1=2(2^m-1)$ and $\#\mathcal{J}_2=(2^m-1)(2^m-2)$. Then by using H\"older's inequality we get
\begin{eqnarray}
\sum_{i\neq j \in \Sigma_m} \mathbb{E}(Y_{n,i}^{q/2}Y_{n,j}^{q/2})&=&\sum_{(i,j)\in\mathcal{J}_1}\mathbb{E}(Y_{n,i}^{q/2}Y_{n,j}^{q/2})+\sum_{(i,j)\in\mathcal{J}_2} \mathbb{E}(Y_{n,i}^{q/2}Y_{n,j}^{q/2})\nonumber \\
&\le& 2(2^m-1)\mathbb{E}(Y_{n,\bar{0}}^q)+\sum_{(i,j)\in\mathcal{J}_2}\mathbb{E}(Y_{n,i}^{q/2}Y_{n,j}^{q/2}),\label{b-1}
\end{eqnarray}
where we denote by $\bar{0}=0\cdots 0\in \Sigma_m$. We need the following lemma:
\begin{lemma}\label{mmon}
There exists a constant $C$ such that for any $(i,j)\in \mathcal{J}_2$ and $q\in(0,1)$,
\[
\mathbb{E}(Y_{n,i}^{q/2}Y_{n,j}^{q/2}) \le C\cdot 2^{(1+\varphi(q))m} \cdot\mathbb{E}\big(\mu^{I_{\bar{0}}}_{2^{-n}}(I_{\bar{0}})^{q/2}\big)^2.
\]
\end{lemma}

This gives
\[
\sum_{(i,j)\in\mathcal{J}_2} \mathbb{E}(Y_{n,i}^{q/2}Y_{n,j}^{q/2}) \le (2^m-1)(2^m-2)\cdot C \cdot 2^{(1+\varphi(q))m} \cdot \mathbb{E}\big(\mu^{I_{\bar{0}}}_{2^{-n}}(I_{\bar{0}})^{q/2}\big)^2.
\]
First notice that $\mu^{I_{\bar{0}}}_{2^{-n}}(I_{\bar{0}})$ has the same law as $Y_{n-m}$. Then combing \eqref{bq} and \eqref{b-1}, and using the fact that $ \mathbb{E}(Y_n^q)\le \mathbb{E}(Y_{n-m}^q)\le 1 $ we get
\[
\mathbb{E}(Y_n^q)\frac{1-e^{-\varphi(q)m\log 2}}{1-q} \le 2+ C(2^m-1)\mathbb{E}(Y_{n-m}^{q/2})^2.
\]
By letting $q\to 1^-$ we obtain
\[
-\varphi'(1^-)m\log 2 \le 2+C(2^m-1) \mathbb{E}(Y_{n-m}^{1/2})^2.
\]
Choose $m$ large enough so that $\varphi'(1^-)m\log 2+2<0$, we get $\inf_{n\ge 1} \mathbb{E}(Y_n^{1/2})>0$. Consequently $\mathbb{E}(Z^{1/2})>0$, thus $\mathbb{E}(Z)>0$. \qed

\subsection{Proof of Lemma \ref{mmon}}\label{pmmon}

The proof can be deduced from \cite[Lemma 3, p. 495-496]{BaMu03}. For reader's convenience we present one here. Write
\[
V_{2^{-n}}^I(t)=V_{2^{-m}}^I(t) \cup V^{m}_{n}(t),
\]
where $V^{m}_{n}(t)=V_{2^{-n}}^I(t)\setminus V_{2^{-m}}^I(t)$. Define the random measure
\[
\mu^m_n(t)=\frac{1}{|I|}\cdot Q(V^{m}_{n}(t))\, dt, \ \ t\in I.
\]
Then for $i\in\Sigma_m$ we have
\[
\mu_{2^{-n}}^I(I_i) \le \left(\sup_{t\in I_i} e^{\Lambda(V_{2^{-m}}^I(t))}\right) \mu^m_n(I_i).
\]
Notice that for $(i,j)\in \mathcal{J}_2$, $\mu^m_n(I_i)$ and $\mu^m_n(I_j)$ are independent, and they are independent of $\sup_{t\in I_i} e^{\Lambda(V_{2^{-m}}^I(t))}$ and $\sup_{t\in I_j} e^{\Lambda(V_{2^{-m}}^I(t))}$. Thus
\begin{eqnarray}
\mathbb{E}(Y_{n,i}^{q/2}Y_{n,j}^{q/2})&\le&\mathbb{E}\left( \prod_{l=i,j}  \sup_{t\in I_l} e^{q\Lambda(V_{2^{-m}}^I(t))/2} \cdot \mu^m_n(I_l)^{q/2}\right)\nonumber\\
&=&\prod_{l=i,j} \mathbb{E}\left( \mu^m_n(I_l)^{q/2}\right) \cdot \mathbb{E}\left( \prod_{l=i,j} \sup_{t\in I_l} e^{q\Lambda(V_{2^{-m}}^I(t))/2}\right)\nonumber\\
&\le&\prod_{l=i,j} \mathbb{E}\left( \mu^m_n(I_l)^{q/2}\right) \cdot\prod_{l=i,j} \mathbb{E}\left( \prod_{l=i,j} \sup_{t\in I_l} e^{q\Lambda(V_{2^{-m}}^I(t))}\right)^{1/2}, \label{mude}
\end{eqnarray}
where the last inequality comes from H\"older's inequality.

Take $J\in\{I_i,I_j\}$ with $J=[t_0,t_1]$. For $t\in J$ we can divide $V_{2^{-m}}^I(t)$ into three disjoint parts:
\begin{equation}\label{vid}
V_{2^{-m}}^I(t)=V^I(J)\cup V^{J,l}(t) \cup V^{J,r}(t),
\end{equation}
where 
\begin{eqnarray*}
V^{J,l}(t)&=&\left\{z=x+iy\in V(t): 2^{-m} \le y < 2(t_1-x)\right\},\\
V^{J,r}(t)&=&\left\{z=x+iy\in V(t): 2^{-m} \le  y \le 2(x-t_0)\right\}.
\end{eqnarray*}
We need the following lemma.

\begin{lemma}\label{Ccq}
Let $s\in\{l,r\}$. For $q\in I_\nu$ there exists constant $C_q<\infty$ such that
\[
\mathbb{E}\left(\sup_{t\in J}e^{q\Lambda(V^{J,s}(t))} \right) \le C_q;
\]
For $q\in\mathbb{R}$ there exists constant $c_q>0$ such that
\[
\mathbb{E}\left(\inf _{t\in J} e^{q\Lambda(V^{J,s}(t))}\right) \ge c_q.
\]
\end{lemma}

By using Lemma \ref{Ccq} we get from \eqref{vid} that for $q\in I_\nu\cap (0,\infty)$,
\begin{equation}\label{qvid}
\mathbb{E}\left( \sup_{t\in J} e^{q\Lambda(V^{I}_{2^{-m}}(t))}\right) \le C_q^2 \cdot \mathbb{E}(e^{q\Lambda(V^{I}(J))})=C_q^2 \cdot 2^{m\psi(-iq)}.
\end{equation}
Also notice that for $t\in J$ we have
\[
V^{m}_n(t)\cup V^{J,l}(t) \cup V^{J,r}(t) =V^{I}_{2^{-n}}(t).
\]
So for any $q'\in \mathbb{R}$ we have
\[
\mu^{J}_{2^{-n}}(J)^{q'} \ge \mu^m_n(J)^{q'} \cdot \left(\inf_{t\in J} e^{q'\Lambda(V^{J,l}(t))} \right) \left(\inf_{t\in J} e^{q'\Lambda(V^{J,r}(t))}\right).
\]
Applying Lemma \ref{Ccq}  we get that
\begin{equation}
\mathbb{E}\left( \mu^m_n(J)^{q/2}\right) \le c_q^{-2}\cdot 2^{-mq/2}\cdot  \mathbb{E}\left(\mu^{J}_{2^{-n}}(J)^{q/2}\right).
\end{equation}
Together with \eqref{mude} and \eqref{qvid} this implies
\[
\mathbb{E}(Y_{n,i}^{q/2}Y_{n,j}^{q/2}) \le C_q^2 c_q^{-2} \cdot 2^{m(1+\varphi(q))}  \cdot \prod_{l=i,j}\mathbb{E}\left(\mu^{I_l}_{2^{-n}}(I_l)^{q/2}\right).
\]
From the prove of Lemma \ref{Ccq} one can chose $C_q c_q^{-1}$ as a increasing function of $q$, and since $1\in I_\nu$, we get the conclusion by taking $C=C_1^2c_1^{-2}$. \qed

\subsubsection{Proof of Lemma \ref{Ccq}}
First let $q\in I_\nu$. We have
\[
\mathbb{E}(\Lambda(V^{J,r}(t)))=a\lambda(V^{J,r}(t)).
\]
From the fact that $\lambda(V^{J,r}(t))=(t-t_0)/|J|$ we get
\[
e^{q\Lambda(V^{J,r}(t))} \le e^{|aq|} \cdot e^{qM_t},
\]
where $M_t=\Lambda(V^{J,r}(t))-a (t-t_1)/|J|$ is a martingale. As $x\mapsto e^{xq/2}$ is convex we have that $e^{qM_t/2}$ is a positive submartingale. Due to Doob's $L^2$-inequality we get
\[
\mathbb{E}\left(\sup_{t\in J}e^{qM_t} \right) \le 4 \sup_{t\in J} \mathbb{E} (e^{qM_t}) \le 4e^{|aq|+|\psi(-iq)|}.
\]
This implies
\[
\mathbb{E}\left(\sup_{t\in J} e^{q\Lambda(V^{J,r}(t))} \right) \le C_q,
\]
where the constant $C_{q}$ only depends on $q$.

Now let $q\in\mathbb{R}$. Notice that
\[
[0,1]\ni t\mapsto \Lambda(V^{J,r}(t_0+(t_1-t_0)t))
\]
is a L\'evy process restricted on $[0,1]$, thus for $X_q=\inf _{t\in J} e^{q\Lambda(V^{J,r}(t))}$ we must have
\[
\mathbb{P}\{X_q > \epsilon_q\}>0
\]
for some $1>\epsilon_q>0$, otherwise this would  contradict the fact that almost surely the sample path of a L\'evy process is c\`adl\`ag. Then
\[
\mathbb{E}\left(\inf _{t\in J} e^{q\Lambda(V^{J,r}(t))}\right) \ge \mathbb{P}\{X_q > \epsilon_q\} \cdot \epsilon_q>0.
\]
The argument for $V^{J,l}(t)$ is the same.\qed

\section{Proof of Theorem \ref{pm}}

We only need to prove that for $q>1$, $0<\mathbb{E}(Z^q)<\infty$ implies that $q\in I_\nu$ and $\varphi(q)<0$, the rest of the result comes from \cite[Lemma 3]{BaMu03}.

Because the function $x^q$ is super-additive, one has
\[
2^qZ^q\ge W_0^qZ_0^q+W_1^qZ_1^q,
\]
and the strict inequality holds if and only if $W_0Z_0=W_1Z_1$. So if $W_0Z_0\neq W_1Z_1$ with positive probability, then
\[
2^q\mathbb{E}(Z^q) >2\mathbb{E}(W_0^q)\mathbb{E}(Z^q),
\]
that is $\mathbb{E}(W_0^q)<2^{q-1}$, which implies that $q\in I_\nu$ and $\varphi(q)<0$. Otherwise $W_0Z_0= W_1Z_1$ almost surely, thus $\varphi(q)=q-1$ for all $q\in I_\nu$. This yields that $\sigma^2=0$ and $\nu\equiv 0$, which is in contradiction to our assumption.

\section{Proof of Theorem \ref{eo}}

\subsection{Proof of (1)} According to Theorem \ref{pm}, $(\alpha)$ implies that $I_\nu\supset [0,\infty)$ and $\varphi(q)<0$ for all $q>1$. Recall that $\varphi(q)=\psi(-iq)-q+1$ and
\[
\psi(-iq)=aq+\frac{1}{2}\sigma^2q^2+ \int_{\mathbb{R}} \bigl(e^{q x}-1-q x \mathbf{1}_{|x|\le1}\bigr)  \nu(\mathrm{d}x).
\]
Suppose that $\nu([\epsilon,\infty))>0$ for some $\epsilon>0$, then one can find constant $c_1,c_2>0$ such that
\[
\psi(-iq)\ge c_1 e^{q\epsilon}-c_2q
\]
as $q\to \infty$, which is in contradiction to $\varphi(q)<0$ for all $q>1$. It is also easy to see that $\varphi(q)<0$ for all $q>1$ implies $\sigma=0$. Thus using the expression of the normalizing constant $a$ (see \eqref{a}) we may write
\begin{equation}\label{vpp'}
\varphi(q)=1-q +\int_{-\infty}^0 \big(e^{qx}-1+q(1-e^x)\big) \, \nu(\mathrm dx).
\end{equation}
It is easy to check that the integral term in \eqref{vpp'} is non-negative, and goes to $\infty$ faster than any multiple of $q$ if $\int_{-\infty}^01\wedge |x|\, \nu(dx)=\infty$, in which case we cannot have $\varphi(q)<0$ for all $q>1$. If $\int_{-\infty}^01\wedge |x|\, \nu(dx)<\infty$, then 
\begin{equation}\label{vpp}
\varphi(q)=(\gamma-1) q+1-\int_{-\infty}^0 \big(1-e^{qx}\big) \, \nu(\mathrm dx),
\end{equation}
where
\[
\gamma=\int_{-\infty}^0 \big(1-e^x\big) \, \nu(\mathrm dx).
\]
Clearly $\varphi(q)<0$ for all $q>1$ implies that $\gamma-1 \le 0$.

Conversely, if $(\beta)$ holds, then $I_\nu\supset [0,\infty)$, since $\nu$ is carried by $(-\infty,0]$ thus $\int_{|x|>1} e^{qx} \nu(\mathrm dx)<\infty$ for any $q>0$. We may write $\varphi(q)$ as in \eqref{vpp}. If $\gamma<1$, then $\lim_{q\to \infty} \varphi(q)=-\infty$ since $\varphi(q)\sim (\gamma-1)q$ at $\infty$.  If $\gamma=1$, then
\[
\int_{-\infty}^0 \big(1-e^{qx}\big) \, \nu(dx)> \int_{-\infty}^0 \big(1-e^{x}\big) \, \nu(\mathrm dx)=\gamma=1
\]
for any $q>1$. Due to the convexity of $\varphi$, it follows that in both cases  $\varphi'(1)<0$ and $\varphi(q)<0$ for all $q>1$, hence we get $(\alpha)$ from Theorem \ref{nd} and Theorem \ref{pm}.

\subsection{Proof of (2)}\label{p5.2} The proof is inspired by the approach used by  Kahane in \cite{KaPe76} for canonical cascades. However, here again the correlations between $Z_0$ and $Z_1$ creates complications. For the sharp upper bound of ${\displaystyle \limsup_{n\to \infty}}\frac{\log \mathbb{E}(Z^n)}{n\log n}$, we use a new approach consisting in writing an explicit formula for the moments of positive integer orders of $Z$  and then estimate them from above by using Dirichlet's multiple integral formula. For the lower bound of ${\displaystyle \liminf_{n\to \infty}}\frac{\log \mathbb{E}(Z^n)}{n\log n}$, we first show that under $(\beta)$ the inequality $\mathbb{E}(\mu(I_0)^k\mu(I_1)^l)\ge \mathbb{E}(\mu(I_0)^k)\mathbb{E}(\mu(I_1)^l)$ holds for any non negative integers $k$ and $l$, and then follow \cite{KaPe76}. 

\medskip

From $(\beta)$ we have that for $q\ge 0$,
\[
\psi(-iq)=\gamma \cdot q-\int_{-\infty}^0 \big(1-e^{qx}\big) \, \nu(\mathrm dx).
\]
We have almost surely
\begin{eqnarray*}
\mu(I)^n&=&\lim_{\epsilon\to 0} \mu_\epsilon(I)^n\\
&=& \lim_{\epsilon\to 0}\left(\int_{t\in I} e^{\Lambda(V_\epsilon^I(t))} \mathrm dt\right)^n.
\end{eqnarray*}
Thus we get from the martingale convergence theorem, Fubini's theorem and dominated convergence theorem that
\[
\mathbb{E}(\mu(I)^n) = \int_{t_1,\cdots, t_{n}\in I} \lim_{\epsilon\to\infty} \mathbb{E}\left( \prod_{j=1}^{n} e^{\Lambda(V_\epsilon^I(t_j))} \right) \mathrm dt_1\cdots \mathrm d t_{n}.
\]
For integers $ k\le j$ define
\begin{eqnarray*}
\alpha(j,k)&=&\psi(-i(j-k+1))+\psi(-i((j-1)-(k+1)+1))\\
&&-\psi(-i((j-1)-k+1))-\psi(-i(j-(k+1)+1))\\
&=&\int_{-\infty}^0 e^{(j-k-1)x}(1-e^x)^2 \, \nu(\mathrm dx).
\end{eqnarray*}
Fix $0<t_1<\cdots <t_n <1$.  Then for $\epsilon$ small enough one gets from \cite[Lemma 1]{BaMu03} that
\[
\log \mathbb{E}\Big(\prod_{j=1}^{n} e^{\Lambda(V_\epsilon^I(t_j))} \Big)=\sum_{k=1}^{n-1}\sum_{j=k+1}^n \alpha(j,k)\cdot \log \frac{1}{t_j-t_k}.
\]
This gives
\[
\mathbb{E}(\mu(I)^n)= n! I_n,
\]
where
\[
I_n=\int_{0<t_1<\cdots<t_n<1} \prod_{k=1}^{n-1}\prod_{j=k+1}^{n} (t_j-t_k)^{-\alpha(j,k)} \mathrm dt_1\cdots \mathrm dt_n.
\]
Let us  use the change of variables $x_1=t_1$ and $x_k=t_k-t_{k-1}$ for $k=2,\cdots,n$. Then $I_n$ becomes
\[
I_n=\int_{x_1+\cdots+x_n\le 1} \prod_{k=1}^{n-1}\prod_{j=k+1}^n\Big(\sum_{l=k+1}^j x_l\Big)^{-\alpha(j,k)} \mathrm dx_1\cdots \mathrm dx_n.
\]
For every integer $l$ define
\[
\gamma_l=\int_{-\infty}^0 e^{lx}(1-e^x)^2 \, \nu(\mathrm dx)
\]
so that
\[
\alpha(j,k)=\gamma_{j-k-1}.
\]
Then we have
\[
 \prod_{k=1}^{n-1}\prod_{j=k+1}^n\Big(\sum_{l=k+1}^j x_l\Big)^{-\alpha(j,k)}=\prod_{l=1}^{n-1} \Big(\prod_{k=1}^{n-l}\big(\sum_{j=k+1}^{k+l}x_j\big)\Big)^{-\gamma_{l-1}}.
\]
Since $x_j\in(0,1)$, it is easy to deduce that for $l=1,\cdots,n-1$,
\[
\prod_{k=1}^{n-l}\big(\sum_{j=k+1}^{k+l}x_j\big) \ge \prod_{j=2}^{n} x_j.
\]
This implies
\[
I_n\le \int_{x_1+\cdots+x_n\le 1} \Big(\prod_{j=2}^{n}x_j\Big)^{-\sum_{l=1}^{n-1} \gamma_{l-1}}\, \mathrm dx_1\cdots \mathrm dx_n.
\]
Notice that
\[
\sum_{l=1}^{n-1} \gamma_{l-1}=\int_{-\infty}^0 (1-e^{(n-1)x})(1-e^x) \, \nu(\mathrm dx):=\gamma'_{n-1}.
\]
Then we get from Dirichlet's multiple integral formula that
\begin{eqnarray*}
&&\int_{x_1+\cdots+x_n\le 1} \Big(\prod_{j=2}^{n}x_j\Big)^{-\gamma'_{n-1}} \, \mathrm dx_1\cdots \mathrm dx_n\\
&=&\int_{x_2+\cdots+x_n\le 1} \Big(1-\sum_{j=2}^nx_j\Big)\cdot \Big(\prod_{j=2}^{n}x_j\Big)^{-\gamma'_{n-1}} \, \mathrm dx_2\cdots \mathrm dx_n\\
&=&\frac{\Gamma(1-\gamma_{n-1}')^{n-1}\Gamma(2)}{\Gamma\big ((n-1)(1-\gamma_{n-1}')+2\big )}.
\end{eqnarray*}
Since $\gamma'_n\to \gamma$ as $n\to \infty$, by applying Stirling's formula we finally get
\[
\limsup_{n\to \infty} \frac{\log \mathbb{E}(Z^n)}{n\log n} \le 1-(1-\gamma)=\gamma.
\]

On the other hand, we have
\begin{equation}\label{binom}
\mu(I)^n=(\mu(I_0)+\mu(I_1))^n=\sum_{m=0}^n \frac{n!}{m!(n-m)!}\mu(I_0)^m\mu(I_1)^{n-m}.
\end{equation}
For $1\le m\le n-1$ we have
\[
\begin{aligned}
\mathbb{E}(\mu(I_0)^m&\mu(I_1)^{n-m})=m!(n-m)!\\
&\int_{0<t_1<\cdots<t_m<1/2<t_{m+1}<\cdots<t_n<1} \prod_{k=1}^{n-1}\prod_{j=k+1}^{n} (t_j-t_k)^{-\alpha(j,k)} \, \mathrm dt_1\cdots \mathrm dt_n.
\end{aligned}
\]
Also
\begin{eqnarray*}
\prod_{k=1}^{n-1}\prod_{j=k+1}^{n} (t_j-t_k)^{-\alpha(j,k)}&=&\prod_{k=1}^{m-1}\prod_{j=k+1}^{m}\prod_{k=1}^{m}\prod_{j=m+1}^{n}\prod_{k=m+1}^{n-1}\prod_{j=k+1}^{n} (t_j-t_k)^{-\alpha(j,k)}\\
&\ge&\prod_{k=1}^{m-1}\prod_{j=k+1}^{m}\prod_{k=m+1}^{n-1}\prod_{j=k+1}^{n} (t_j-t_k)^{-\alpha(j,k)},
\end{eqnarray*}
where the inequality uses the fact that $t_j-t_k\le 1$ and $\alpha(j,k)\ge 0$. This implies that
\[
\mathbb{E}(\mu(I_0)^m\mu(I_1)^{n-m})\ge \mathbb{E}(\mu(I_0)^m)\mathbb{E}(\mu(I_1)^{n-m}).
\]
Notice that
\[
\mathbb{E}(\mu(I_0)^m)=2^{-m}\mathbb{E}(W_0^m)\mathbb{E}(Z^m)=2^{-m}2^{\psi(-im)}\mathbb{E}(Z^m).
\]
Since
\[
\psi(-im)=\gamma m-\int_{-\infty}^0 \big(1-e^{mx}\big) \, \nu(\mathrm dx),
\]
for any $\epsilon>0$ there exists $c>0$ such that for all $m\ge 0$ we have 
$$
\psi(-im)\ge (\gamma -\epsilon) m +\log (c),
$$
and using \eqref{binom}
\begin{eqnarray*}
\mathbb{E}(Z^n) &\ge& c^2 2^{(\gamma-\epsilon) n} \sum_{m=0}^n \frac{n!}{m!(n-m)!}2^{-n}\mathbb{E}(Z^m)\mathbb{E}(Z^{n-m})\\
&\ge& c^2 2^{(\gamma-\epsilon) n}  \mathbb{E}(Z^{n/2})^2.
\end{eqnarray*}
Hence
\[
\log \mathbb{E}(Z^{2n})\ge 2\log (c)+ (\gamma-\epsilon) 2n\log2+2\log \mathbb{E}(Z^{n}).
\]
Consequently, 
\begin{eqnarray*}
\frac{\log \mathbb{E}(Z^{2^n})}{2^n} &\ge&\frac{2\log (c)}{2^n} +(\gamma-\epsilon)  \log2+\frac{\log \mathbb{E}(Z^{2^{n-1}})}{2^{n-1}}\\
&\ge& n(\gamma-\epsilon) \log 2 + 2(1-2^{-n})\log (c).
\end{eqnarray*}
This easily yields
\[
\liminf_{n\to \infty}\frac{\log \mathbb{E}(Z^n)}{n\log n}\ge \gamma-\epsilon,
\]
for any $\epsilon>0$.

\section{Proof of Theorem \ref{tail}}

\subsection{Reduction to a key proposition}\label{heart of the proof}
In the case of limits of canonical cascades,  Guivarc'h \cite{Gu90} exploited  \eqref{FEC} to connect our problem to a random difference equation one; then Liu \cite{Liu00}  extended this idea for the case of supercritical Galton-Watson trees, and for this he used explicitly Peyri\`ere's measure. This is our starting point, the difference being that now we must exploit the more delicate equation \eqref{fek}.

Recall that $\pi(\mathbf{i})=\sum_{j=1}^\infty i_j 2^{-j}$ is a continuous map from $\Sigma$ to $[0,1)$. We shall use the same notation $\mu$ for the pull-back measure $\mu\circ \pi^{-1}$ on $\Sigma$. Let $\Omega'=\Omega\times \Sigma$ be the product space, let $\mathcal{F}'=\mathcal{F}\times \mathcal{B}$ be the product $\sigma$-algebra, and let $\mathbb{Q}$ be the Peyri\`ere measure on $(\Omega', \mathcal{F}')$, defined as
\[
\mathbb{Q}(E)=\mathbb{E} \left(\int_\Sigma \mathbf{1}_E(\omega, \mathbf{i}) \, \mu(\mathrm d\mathbf{i}) \right), \ \ E\in \mathcal{F}'.
\]
Then $(\Omega',\mathcal{F}',\mathbb{Q})$ forms a probability space.

For $\omega\in\Omega$ and $\mathbf{i}\in\Sigma$ let
\begin{eqnarray*}
A(\omega,\mathbf{i}) &=& \sum_{i\in\{0,1\}} 2^{-1}W_i(\omega)\cdot \mathbf{1}_{\{\mathbf{i}|_1=i\}},\\
B(\omega,\mathbf{i}) &=& \sum_{i\in\{0,1\}} 2^{-1} W_i(\omega) Z_i(\omega) \cdot \mathbf{1}_{\{\mathbf{i}|_1=1-i\}},\\
R(\omega,\mathbf{i}) &=& \sum_{i\in\{0,1\}} Z_i(\omega)\cdot \mathbf{1}_{\{\mathbf{i}|_1=i\}},\\
\widetilde{R}(\omega,\mathbf{i}) &=& Z(\omega).
\end{eqnarray*}
We may consider $A$, $B$, $R$ and $\widetilde{R}$ as random variables on $(\Omega',\mathcal{F}',\mathbb{Q})$, and we have the following equation
\[
\widetilde{R}=AR+B.
\]
First we claim that $R$ and $\widetilde{R}$ have the same law. This is due to the fact that for any non-negative Borel function $f$ we have
\begin{eqnarray*}
\mathbb{E}_\mathbb{Q}(f(R)) &=& \mathbb{E}\left(2^{-1}\sum_{i\in \{0,1\}} f(Z_i) \cdot W_{i} \cdot Z_i  \right)\\
&=& \mathbb{E}(f(Z)Z) \\
&=& \mathbb{E}_\mathbb{Q}(f(\widetilde{R})).
\end{eqnarray*}
Then we claim that $A$ and $R$ are independent, since for any non-negative Borel functions $f$ and $g$ we have
\begin{eqnarray*}
\mathbb{E}_\mathbb{Q}(f(A)g(R)) &=& \mathbb{E}\left(2^{-1}\sum_{i\in \{0,1\}^n} f(W_{i})g(Z_i)\cdot W_i \cdot Z_i  \right)\\
&=& \mathbb{E}(f(W_0)W_0)\mathbb{E}(g(Z_0)Z_0) \\
&=& \mathbb{E}_\mathbb{Q}(f(A))\mathbb{E}_\mathbb{Q}(g(R)).
\end{eqnarray*}

We first deal with case (i). The following result comes from the implicit renewal theory of random difference equations given by Goldie in \cite{Gol91} (Lemma 2.2, Theorem 2.3 and Lemma 9.4).

\begin{theorem}\label{irt} Suppose there exists $\kappa>0$ such that
\begin{equation}
\mathbb{E}_\mathbb{Q}(A^\kappa)=1, \ \ \mathbb{E}_\mathbb{Q}(A^\kappa \log^+ A) <\infty,
\end{equation}
and suppose that the conditional law of $\log A$, given $A\neq 0$, is non-arithmetic. For
\[
\widetilde{R}=AR+B,
\]
where $\widetilde{R}$ and $R$ have the same law, and $A$ and $R$ are independent, we have that if
\[
\mathbb{E}_\mathbb{Q}\left((AR+B)^\kappa-(AR)^\kappa\right)<\infty,
\]
then
\[
\lim_{t\to \infty } t^\kappa\mathbb{Q}(R>t)=\frac{\mathbb{E}_\mathbb{Q}\left((AR+B)^\kappa-(AR)^\kappa\right)}{\kappa\mathbb{E}_\mathbb{Q}(A^\kappa \log A)} \in(0,\infty).
\]
\end{theorem}
It is worth mentioning that the independence between $B$ and $R$ is not necessary, while in dealing with classical random difference equations it holds systematically and simplifies the verification of crucial assumptions. In our study,  it is crucial that $B$ and $R$ do not need to be independent because the situation for log-infinitely divisible cascades presents much more correlations to control than the case of canonical cascades on homogeneous or Galton-Watson trees. 

For $q\in I_\nu$ we have
\[
\mathbb{E}_\mathbb{Q}(A^{q-1})=2^{1-q}\mathbb{E}(W_0^q)=2^{\varphi(q)}.
\]
Take $\kappa=\zeta-1$ then we get $\mathbb{E}_\mathbb{Q}(A^\kappa)=1$. From $\varphi'(\zeta)<\infty$ it is easy to deduce that $\mathbb{E}_\mathbb{Q}(A^\kappa\log^+ A)<\infty$. In case (i) we have either $\sigma\neq 0$ or $\nu$ is not of the form $\sum_{n\in\mathbb{Z}} p_n \delta_{nh}$ for some $h>0$ and $p_n\ge 0$, thus the conditional law of $\log A$, given $A\neq 0$, is non-arithmetic. So in order to apply Theorem \ref{irt}, it is only left to verify that
$
\mathbb{E}_\mathbb{Q}\left((AR+B)^\kappa-(AR)^\kappa\right)<\infty.
$
To do so, we need the following proposition (in the framework of canonical cascades such a fact is simple to establish due to the independences associated with the branching property (see \cite[Lemma 4.1]{Liu00})).

\begin{proposition}\label{1kappa}
$\mathbb{E}(\mu(I_0)\mu(I_1)^{\kappa})<\infty$.
\end{proposition}

We have
\[
\mathbb{E}_\mathbb{Q}\left((AR+B)^\kappa-(AR)^\kappa\right)=2\mathbb{E}\left((\mu(I)^\kappa-\mu(I_0)^\kappa)\cdot \mu(I_0)\right).
\]
By using the following inequality
\[
(x+y)^\kappa-x^\kappa \le \left\{ \begin{array}{ll}
y^\kappa, & 0<\kappa\le 1,\\
\kappa2^{\kappa-1} y (x^{\kappa-1}+y^{\kappa-1}), & 1<\kappa<\infty.
\end{array}
\right.\ \ x,y>0,
\]
it is easy to find a constant $C_\kappa$ such that
\[
\mathbb{E}_\mathbb{Q}\left((AR+B)^\kappa-(AR)^\kappa\right) \le C_\kappa \mathbb{E}(\mu^I(I_0)\mu^I(I_1)^{\kappa}).
\]
Then from Proposition \ref{1kappa} we get $\mathbb{E}_\mathbb{Q}\left((AR+B)^\kappa-(AR)^\kappa\right)<\infty$.

We have verified all the assumptions in Theorem \ref{irt}, thus
\[
\lim_{t\to \infty } t^\kappa\mathbb{Q}(R>t)=\frac{\mathbb{E}_\mathbb{Q}\left((AR+B)^\kappa-(AR)^\kappa\right)}{\kappa\mathbb{E}_\mathbb{Q}(A^\kappa \log A)}=d' \in(0,\infty).
\]
Notice that $\mathbb{Q}(R>t)=\int_t^\infty x \, \mathbb{P}(Z\in dx)$. From \cite[Lemma 4.3]{Liu00} we get
\[
\lim_{t\to \infty } t^\zeta \mathbb{P}(Z>t) =\frac{d' (\zeta-1)}{\zeta}.
\]
It is easy to verify that
\[
d'=\frac{2\mathbb{E}\left(\mu(I)^{\zeta-1}\mu(I_0)-\mu(I_0)^{\zeta}\right)}{(\zeta-1) \varphi'(\zeta)\log 2},
\]
and this gives the conclusion.

For case (ii), we may apply the key renewal theorem in the arithmetic case instead of the non-arithmetic case used in Goldie's proof of Theorem 2.3, Case 1 (\cite[page 145, line 21]{Gol91}) to get that for $x\in\mathbb{R}$,
\[
\check{r}(x+nh)\to d(x), \ n\to \infty,
\]
where $0<d(x)<\infty$, $r(t)=e^{\kappa t}\mathbb{Q}(R>e^t)$ and
\[
\check{r}(x)=\int_{-\infty}^x e^{-(x-t)} r(t) \, \mathrm dt.
\]
We have for $x+h>y$,
\begin{eqnarray*}
\check{r}(x+h)-\check{r}(y)&=&\int_{0}^{e^{x+h}} e^{-(x+h)}u^\kappa\cdot \mathbb{Q}(R>u) \, \mathrm du-\int_{0}^{e^{y}} e^{-y}u^\kappa\cdot \mathbb{Q}(R>u) \, \mathrm du\\
&=& \frac{e^{-(x+h)}-e^{-y}}{e^{-y}}\check{r}(y) +e^{-(x+h)} \int_{e^y}^{e^{x+h}}u^\kappa\cdot \mathbb{Q}(R>u) \, \mathrm du,
\end{eqnarray*}
thus
\[
\check{r}(x+h)-e^{y-x-h}\check{r}(y)=e^{-(x+h)} \int_{e^y}^{e^{x+h}}u^\kappa\cdot \mathbb{Q}(R>u) \, \mathrm du.
\]
On one hand we have
\begin{eqnarray*}
 e^{-(x+h)} \int_{e^y}^{e^{x+h}}u^\kappa\cdot \mathbb{Q}(R>u) \, \mathrm du&\le& e^{-(x+h)} \cdot e^{(x+h)\kappa}\cdot \mathbb{Q}(R>e^y)\cdot  (e^{x+h}-e^y)\\
&=&(1-e^{y-x-h})\cdot e^{(x+h)\kappa}\cdot \mathbb{Q}(R>e^y).
\end{eqnarray*}
This gives that
\[
\liminf_{n\to \infty}e^{(y+nh)\kappa}\cdot \mathbb{Q}(R>e^{y+nh}) \ge e^{-(x+h-y)\kappa}(1-e^{y-x-h})^{-1}[d(x)-e^{y-x-h}d(y)].
\]
On the other hand we have
\begin{eqnarray*}
 e^{-(x+h)} \int_{e^y}^{e^{x+h}}u^\kappa\cdot \mathbb{Q}(R>u) \, \mathrm du&\ge& e^{-(x+h)} \cdot e^{y\kappa}\cdot \mathbb{Q}(R>e^{x+h})\cdot  (e^{x+h}-e^y)\\
&=&(1-e^{y-x-h})\cdot e^{y\kappa}\cdot \mathbb{Q}(R>e^{x+h}).
\end{eqnarray*}
This gives
\[
\limsup_{n\to \infty}e^{(x+nh)\kappa}\cdot \mathbb{Q}(R>e^{x+nh}) \le e^{(x+h-y)\kappa}(1-e^{y-x-h})^{-1}[d(x)-e^{y-x-h}d(y)].
\]
From these two estimation we can get the conclusion by using the same arguments as in Lemma 4.3(ii) and Theorem 2.2 in \cite{Liu00}. \qed

\subsection{Proof of Proposition \ref{1kappa}}

We have almost surely
\begin{eqnarray*}
\mu(I_0)\mu(I_1)^\kappa&=&\lim_{\epsilon\to 0} \mu_\epsilon(I_0)\mu_\epsilon(I_1)^\kappa\\
&=& \lim_{\epsilon\to 0}\left(\int_{t\in I_0} e^{\Lambda(V_\epsilon^I(t))} \, \mathrm dt\right)\cdot \left(\int_{t\in I_1} e^{\Lambda(V_\epsilon^I(t))} \, \mathrm dt\right)^\kappa.
\end{eqnarray*}
Let $n\ge 1$ be an integer such that $n-1< \kappa\le n$, so $q=\kappa-n+1\in (0,1]$. Thus
\[
\left(\int_{t\in I_1} e^{\Lambda(V_\epsilon^I(t))} \, \mathrm dt\right)^\kappa=\left(\int_{t\in I_1} e^{\Lambda(V_\epsilon^I(t))} \, \mathrm dt\right)^{n-1}\left(\int_{t\in I_1} e^{\Lambda(V_\epsilon^I(t))} \, \mathrm dt\right)^q
\]
Then we get from Fatou's lemma and Fubini's theorem that
\begin{equation}\label{kappa}
\begin{aligned}
&\mathbb{E}(\mu(I_0)\mu(I_1)^\kappa) \le \liminf_{\epsilon\to\infty} \\
&\int_{t_0\in I_0,t_1,\cdots,t_{n-1}\in I_1}\mathbb{E}\left( \prod_{k=0}^{n-1} e^{\Lambda(V_\epsilon^I(t_k))} \cdot \Big[\int_{1/2}^1e^{\Lambda(V_\epsilon^I(t_n))}\, \mathrm dt_n \Big]^q\right) \, \mathrm dt_0 \cdots \mathrm d t_{n-1}.
\end{aligned}
\end{equation}
Denote by $s_0=1/2$, $s_{n}=1$ and $s_1<\cdots<s_{n-1}$ the permutation of $t_1,\cdots,t_{n-1}$. Then from the sub-additivity of $x\mapsto x^q$ we get
\[
\Big[\int_{1/2}^1e^{\Lambda(V_\epsilon^I(t_n))}\, \mathrm dt_n \Big]^q \le \sum_{j=0}^{n-1} \Big[\int_{s_j}^{s_{j+1}}e^{\Lambda(V_\epsilon^I(t_n))}\, \mathrm dt_n \Big]^q.
\]
For each $j=0,\cdots,n-1$ we have
\begin{eqnarray*}
&&\Big[\int_{s_j}^{s_{j+1}}e^{\Lambda(V_\epsilon^I(t_n))}\, \mathrm dt_n \Big]^q\\
&\le&\sup_{s_j<t<s_{j+1} }e^{q\Lambda(V_\epsilon^I(t)\cap V_\epsilon^I(t_0))}\cdot \Big[\int_{s_j}^{s_{j+1}}e^{\Lambda(V_\epsilon^{I}(t_n)\setminus V^I_\epsilon(t_0))}\, \mathrm dt_n \Big]^q\\
&\le&\sup_{s_j<t<s_{j+1} }e^{q\Lambda(V_\epsilon^I(t)\cap V_\epsilon^I(t_0))}\cdot \Big[1+\int_{s_j}^{s_{j+1}}e^{\Lambda(V_\epsilon^{I}(t_n)\setminus V^I_\epsilon(t_0))}\, \mathrm dt_n \Big],
\end{eqnarray*}
where we have used the elementary inequality $x^q\le 1+x$ for $x>0$ and $q\in(0,1]$.
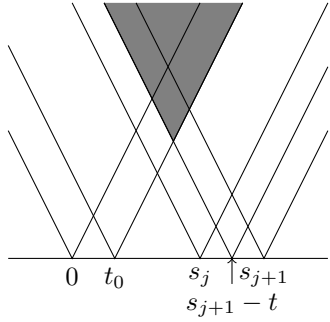
\begin{figure}[ht]
\centering
\begin{tikzpicture}[xscale=1.7,yscale=1.7]
\draw [fill, fill=gray] (1/4,2) -- (1/3+11/24,11/12)  -- (4/3,2);
\draw (-1/2,0) -- (2,0);
\draw (-1/2,5/3) -- (1/3,0) -- (4/3,2); 
\draw (-1/2,1) -- (0,0) -- (1,2);
\draw (0,2) -- (1,0) -- (2,2);
\draw (1/2,2) -- (3/2,0) -- (2,1);
\draw (1/4,2) -- (5/4,0) -- (2,3/2););
\node [below] at (0,0) {$0$};
\node [below] at (1/3,0) {$t_0$};
\node [below] at (1,0) {$s_j$};
\node [below] at (3/2,0) {$s_{j+1}$};
\node [below] at (5/4,-0.2) {$s_{j+1}-t$};
\draw [->] (5/4,-0.2) -- (5/4,0);
\end{tikzpicture}
\caption{The gray area for $V_\epsilon^{I}(s_{j+1}-t)\cap V_\epsilon^I(t_0)$.}\label{figure4}
\end{figure}
For $t\in [0,s_{j+1}-s_j]$ define process $Y_{t}= e^{q\Lambda(V_\epsilon^I(s_{j+1}-t)\cap V_\epsilon^I(t_0))}$ and its natural filtration $\mathcal F_t=\sigma(Y_s: 0\le s\le t )$. Notice that the set $V_\epsilon^I(s_{j+1}-t)\cap V_\epsilon^I(t_0)$ is increasing with respect to $t$ (see Figure \ref{figure4}), thus we actually have
\[
\mathcal F_t=\sigma ( \Lambda (V_\epsilon^{I}(s_{j+1}-t)\cap V_\epsilon^I(t_0))).
\]
For $\eta\in\{0,1\}$ define $D_\eta=e^{\eta\Lambda(V_\epsilon^{I}(t_n)\setminus V^I_\epsilon(t_0))} \prod_{k=0}^{n-1} e^{\Lambda(V_\epsilon^I(t_k))}$. Under the probability $\mathrm{d}\mathbb P_\eta=\frac{D_\eta}{\mathbb E(D_\eta)}\mathrm{d}\mathbb P$ we have the following two facts: (1) $t\mapsto \mathbb{E}_{\mathbb{P}_\eta}(Y_{t})$ is continuous; (2) $Y_t$ is a positive submartingale with respect to $\mathcal{F}_{t}$. The continuity and positivity are obvious, so we only need to verify the following: for $0<s<s+\epsilon<s_{j+1}$ if we write $\Delta_{s,\epsilon}=(V_\epsilon^I(s_{j+1}-t-\epsilon)\setminus V_\epsilon^I(s_{j+1}-t ))\cap V_\epsilon^I(t_0)$ and let $m$ be the corresponding power of $e^{\Lambda(\Delta_{s,\epsilon})}$ appeared in $D_\eta$, then we have
\begin{eqnarray*}
\mathbb{E}_{\mathbb{P}_\eta}(Y_{s}| \mathcal{F}_{s}) &=& e^{(\psi(-i(q+m))-\psi(-im)) \lambda(\Delta_{s,\epsilon})} \cdot \mathbb{E}_{\mathbb{P}_\eta}(Y_{s}| \mathcal{F}_{s}) \\
&\ge& \mathbb{E}_{\mathbb{P}_\eta}(Y_{s}| \mathcal{F}_{s}),
\end{eqnarray*}
where the inequality comes from the fact that $\psi(-ip)$ is an increasing function of $p$ on the right of $1$ since it is convex and $\frac{d}{dp} \psi(-ip)|_{p=1}>0$. When $t_0,s_j, s_{j+1}$ are fixed, we have $\sup_{0<t<s_{j+1}-s_j} \mathbb{E}_{\mathbb{P}_\eta}(Y_{t})<\infty$, thus almost every path of $Y_t$ is c\`adl\`ag (see \cite[Proposition 2.6, Theorem 2.8]{ReYo99} for example). Then Doob's inequality applied with $L^{\gamma}$ ($\gamma>1$) yields $c=c(\gamma)$ such that 
\begin{multline*}
\mathbb E\left (e^{\eta\Lambda(V_\epsilon^{I}(t_n)\setminus V^I_\epsilon(t_0))} \Big (\prod_{k=0}^{n-1} e^{\Lambda(V_\epsilon^I(t_k))}\Big ) \sup_{s_j<t<s_{j+1} }e^{q\Lambda(V_\epsilon^I(t)\cap V_\epsilon^I(t_0))}\right )\\
\le  c \mathbb{E}(D_\eta) ^{1-1/\gamma}\left[ \mathbb E\left (e^{\eta\Lambda(V_\epsilon^{I}(t_n)\setminus V^I_\epsilon(t_0))} \Big (\prod_{k=0}^{n-1} e^{\Lambda(V_\epsilon^I(t_k))}\Big ) e^{q\gamma\Lambda(V_\epsilon^I(s_j)\cap V_\epsilon^I(t_0))}\right )\right]^{1/\gamma}.
\end{multline*}
Thus
\begin{eqnarray*}
&&\mathbb{E}\left( \prod_{k=0}^{n-1} e^{\Lambda(V_\epsilon^I(t_k))} \cdot \Big[\int_{s_j}^{s_{j+1}}e^{\Lambda(V_\epsilon^I(t_n))}\, dt_n \Big]^q\right)\\
&\le& c\mathbb{E}(D_0)^{1-1/\gamma}\left[  \mathbb{E}\left( \prod_{k=0}^{n-1} e^{\Lambda(V_\epsilon^I(t_k))} \cdot e^{q\gamma\Lambda(V_\epsilon^I(s_j)\cap V_\epsilon^I(t_0))}\right )\right]^{1/\gamma} + c \mathbb{E}(D_1)^{1-1/\gamma} \cdot \\
&& \int_{s_j}^{s_{j+1}}\left[   \mathbb E\left (e^{\Lambda(V_\epsilon^{I}(t_n)\setminus V^I_\epsilon(t_0))} \Big (\prod_{k=0}^{n-1} e^{\Lambda(V_\epsilon^I(t_k))}\Big ) e^{q\gamma\Lambda(V_\epsilon^I(s_j)\cap V_\epsilon^I(t_0))}\right )\right]^{1/\gamma}\, \mathrm dt_n .
\end{eqnarray*}
For $\eta,\eta'\in\{0,1\}$ and $t_n\in [s_j,s_{j+1})$ define
\[
\widetilde{\Lambda}_{\eta,\eta'}(t_n)=\begin{cases}
q\gamma\eta'\Lambda(V_\epsilon^I(s_j)\cap V_\epsilon^I(t_0))+\eta\Lambda(V_\epsilon^{I}(t_n)\setminus V^I_\epsilon(t_0))& \text{  if } q<1,\\
\Lambda(V_\epsilon^I(t_n)) &\text{ if } q=1.
\end{cases}
\]
Then define
\[
\overline{D}_{\eta,\eta'}(t_0,\cdots,t_n)=\mathbb{E}\left( \prod_{j=0}^{n-1} e^{\Lambda(V_\epsilon^I(t_j))} \cdot e^{\widetilde{\Lambda}_{\eta,\eta'} (t_n)}\right).
\]
It is easy to see that $\mathbb{E}(D_0)=\overline{D}_{0,0}(t_0,\cdots,t_n)$, $\mathbb{E}(D_1)=\overline{D}_{1,0}(t_0,\cdots,t_n)$, 
\[
\mathbb{E}\left( \prod_{k=0}^{n-1} e^{\Lambda(V_\epsilon^I(t_k))} \cdot e^{q\gamma\Lambda(V_\epsilon^I(s_j)\cap V_\epsilon^I(t_0))}\right )=\overline{D}_{0,1}(t_0,\cdots,t_n)
\]
and
\[
\mathbb E\left (e^{\Lambda(V_\epsilon^{I}(t_n)\setminus V^I_\epsilon(t_0))} \Big (\prod_{k=0}^{n-1} e^{\Lambda(V_\epsilon^I(t_k))}\Big ) e^{q\gamma\Lambda(V_\epsilon^I(s_j)\cap V_\epsilon^I(t_0))}\right )=\overline{D}_{1,1}(t_0,\cdots,t_n).
\]
Also set $\gamma_q=\gamma$ if $q<1$ and $\gamma_q=1$ if $q=1$. We finally get
\begin{eqnarray*}
&&\mathbb{E}\left( \prod_{j=0}^{n-1} e^{\Lambda(V_\epsilon^I(t_j))} \cdot \Big[\int_{1/2}^1e^{\Lambda(V_\epsilon^I(t_n))}\, dt_n \Big]^q\right)  \\
&\le&  2c\cdot  \sum_{\eta\in\{0,1\}}\mathbb{E}(D_{\eta,0})(t_0,\cdots,t_n)^{1-1/\gamma_q}\int_{1/2}^1\overline{D}_{\eta,1}(t_0,\cdots,t_n)^{1/\gamma_q} \mathrm dt_n\\
&\le& 4c\cdot \int_{1/2}^1\max_{\eta,\eta'\in\{0,1\}} \overline{D}_{\eta,\eta'}(t_0,\cdots,t_n) \, \mathrm dt_n.
\end{eqnarray*}

Now fix $t_0,\cdots,t_n$ and redefine $s_0=t_0$, $s_1=1/2$ and $s_2<\cdots<s_{n+1}$ the permutation of $t_1,\cdots,t_n$. Let $j_*$ be such that $s_{j_*}=t_n$. Define
\[
\left\{\begin{array}{ll}
p_0=1; & \\
p_1=0; & \\
p_j=1, & \text{for } j\neq j_*;\\
p_{j_*}=\eta, & \text{in case of } q<1 ;\\
p_{j_*}=1, & \text{in case of } q=1 .\\
\end{array}
\right.
\]
For $k=0,\cdots,n$ and $j=k,\cdots,n+1$ define
\[
r_{k,j}=\left\{
\begin{array}{ll}
q\gamma\eta'+\sum_{l=k,\cdots,j; s_j\neq t_n} p_l, & \text{if } q<1, k=0 \text{ and } t_n\in\{s_{j},s_{j+1}\};\\
\sum_{l=k,\cdots,j} p_l, & \text{otherwise}.
\end{array}
\right.
\]
and let $r_{k,j}=0$ for $k<j$. Then by using the same argument as \cite[Lemma 1]{BaMu03} (notice that $r_{k,j}$ represents the power to $e^{V^I_\epsilon(s_k)\cap V^I_\epsilon(s_j) \setminus (V^I_\epsilon(s_{k-1})\cup V^I_\epsilon(s_{j+1}))}$ which appears in the product $\prod_{j=0}^{n-1} e^{\Lambda(V_\epsilon^I(t_j))} \cdot e^{\widetilde{\Lambda}_{\eta,\eta'} (t_n)}$, and that $\lambda \big (V^I_\epsilon(s_k)\cap V^I_\epsilon(s_j) \setminus (V^I_\epsilon(s_{k-1})\cup V^I_\epsilon(s_{j+1}))\big)=\log \frac{1}{s_j-s_k}+\log \frac{1}{s_{j+1}-s_k}-\log \frac{1}{s_j-s_{k-1}}-\log \frac{1}{s_{j+1}-s_{k-1}}$, see Figure \ref{figure2}) we can get
\[
\overline{D}_{\eta,\eta'}(t_0,\cdots,t_n)= \sum_{k=0}^{n}\sum_{j=k+1}^{n+1} \alpha(j,k)\cdot \log \frac{1}{s_j-s_k},
\]
where
\[
\alpha(j,k)=\psi(-i r_{k,j})+\psi(-i r_{k+1,j-1})-\psi(-ir_{k,j-1})-\psi(-ir_{k+1,j}).
\]
\begin{figure}[ht]
\centering
\begin{tikzpicture}[xscale=1.7,yscale=1.7]
\draw (-1/2,0) -- (2,0);
\draw (-1/2,5/3) -- (1/3,0) -- (4/3,2); 
\draw (-1/2,1) -- (0,0) -- (1,2);
\draw (0,2) -- (1,0) -- (2,2);
\draw (1/2,2) -- (3/2,0) -- (2,1);
\draw [fill, fill=gray] (3/4,3/2) -- (1/2,1) -- (2/3,2/3) -- (11/12,7/6) -- (3/4,3/2);
\node [below] at (0,0) {$s_{k-1}$};
\node [below] at (1/3,0) {$s_k$};
\node [below] at (1,0) {$s_j$};
\node [below] at (3/2,0) {$s_{j+1}$};
\node [below] at (2/3,0) {$\cdots$};
\end{tikzpicture}
\caption{$r_{k,j}$ is the power corresponding to the gray area.}\label{figure2}
\end{figure}
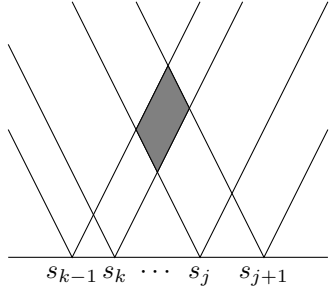
Let $\widetilde\psi(p)=\psi(-ip)$. By definition of $\kappa$, we have $\widetilde\psi(p)<p-1$ for all $p\in (1,n+q)$, and $\widetilde\psi(n+q)=n+q-1$. Moreover, $\widetilde \psi'(1)<1$ since $\varphi'(1)<0$, and $\widetilde\psi(1)=0$. Consequently, there exists $\delta\in (0,1)$ such $\widetilde \psi (p)\le (1-\delta)(p-1)$ for $p\in[1,n]$; in particular by convexity of $\widetilde\psi$ we have $1-\delta\ge \widetilde\psi'(1)$. Moreover, notice that $\widetilde\psi(p)\le 0$ for $p\in (0,1)$ since $\widetilde\psi(0)=0=\widetilde\psi(1)$ and $\widetilde\psi$ is convex, and also $\widetilde \psi(p)\ge \widetilde\psi'(1) (p-1)$ for all $p\ge 0$, which yields for $p\in [0,1]$, $\widetilde \psi(p)  \ge (1-\delta)(p-1)$. Finally, in case of $q<1$, we take $\gamma>1$ small enough such  that $q\gamma<1$ and $\widetilde{\psi}(n+q\gamma)-n+1=q'<1$.
 
\medskip

(i) If $n=1$, that is $0< \kappa\le 1$, $q=\kappa$ and $\widetilde\psi(1+q\gamma)=q'<1$. We have $s_0=t_0\in [0,1/2)$, $s_1=1/2$, $s_2=t_1\in [1/2,1)$ and $s_3=1$.

If $q<1$, we have
\[
r_{0,0}=1,\ r_{0,1}=1+q\gamma\eta',\ r_{0,2}=1+q\gamma\eta', \ r_{1,1}=0,\ r_{1,2}=\eta,\ r_{2,2}=\eta.
\]
This gives
\begin{eqnarray*}
\alpha(0,1)&=&\widetilde\psi(1+q\gamma\eta')+\widetilde\psi(0)-\widetilde\psi(1)-\widetilde\psi(0)\le q',\\
\alpha(0,2)&=&\widetilde\psi(1+q\gamma\eta')+\widetilde\psi(0)-\widetilde\psi(1+q\gamma\eta')-\widetilde\psi(\eta)=0,\\
\alpha(1,2)&=&\widetilde\psi(\eta)+\widetilde\psi(0)-\widetilde\psi(0)-\widetilde\psi(\eta)=0.
\end{eqnarray*}
Thus
\[
\mathbb{E}(\mu(I_0)\mu(I_1)^\kappa)\le 4c\cdot \int_{0}^{1/2} (1/2-s)^{-q'} \, ds <\infty.
\]

If $q=1$, we have
\[
r_{0,0}= r_{0,1}=r_{1,1}= r_{1,2}=1,\ r_{0,2}=2.
\]
This gives $\alpha(0,1)=\alpha(1,2)=0$ and $\alpha(0,2)=\widetilde\psi(2)=1$. Thus
\[
\mathbb{E}(\mu(I_0)\mu(I_1))= \int_{0}^{1/2}\int_{1/2}^1 (t_1-t_0)^{-1} \, \mathrm dt_0 \mathrm dt_1 =\log 2<\infty.
\]
\begin{remark}\label{tr1.2}
Here we have an equality since when $q$ is an integer we do not need to use Doob's inequality to estimate \eqref{kappa} and we can apply the martingale convergence theorem and dominated convergence theorem as in Section \ref{p5.2}. The identity $\mathbb{E}(\mu(I_0)\mu(I_1))=\log 2$ yields the precise formula in Remark \ref{r1.2}.
\end{remark}

\medskip

(ii) The case $n\ge 2$ is more involved. For $0\le k<j\le n+1$, write 
$$
\alpha(j,k)=\beta(j,k)-\beta(j,k+1),\text{ where } \beta(j,k)=\widetilde\psi(r_{k,j})-\widetilde\psi(r_{k,j-1}).
$$
Then 
\begin{eqnarray*}
\sum_{k=0}^{n}\sum_{j=k+1}^{n+1} \alpha(j,k)\cdot \log \frac{1}{s_j-s_k}&=& \sum_{k=0}^{n}\sum_{j=k+1}^{n+1} (\beta(j,k)-\beta(j,k+1))\cdot \log \frac{1}{s_j-s_k}\\
&=& \sum_{j=1}^{n+1}\sum_{k=0}^{j-1}(\beta(j,k)-\beta(j,k+1))\cdot \log \frac{1}{s_j-s_k}\\
&=& \widetilde A+\widetilde B+\widetilde C, 
\end{eqnarray*}
where 
\begin{align*}
\widetilde A&=\sum_{j=1}^{n+1}\sum_{k=0}^{j-1}\beta(j,k)\cdot \log \frac{s_j-s_{k-1}}{s_j-s_k},\\ 
\widetilde B&=\sum_{j=1}^{n+1}\beta (j,0)\cdot \log \frac{1}{s_j-s_0},\ \widetilde C=-\sum_{j=1}^{n+1}\beta (j,j)\cdot \log \frac{1}{s_j-s_{j-1}}.
 \end{align*}
 Now, using the definition of $\beta(j,k)$ we get 
\begin{align*}
\widetilde A&=\sum_{k=1}^{n}\sum_{j=k+1}^{n+1}\beta(j,k)\cdot \log \frac{s_j-s_{k-1}}{s_j-s_k},\\ 
&=\sum_{k=1}^{n}\widetilde\psi(r_{k,n+1})\cdot \log \frac{s_{n+1}-s_{k-1}}{s_{n+1}-s_k}\\
&\quad\quad\quad\quad\quad\quad\quad\quad + \sum_{k=1}^{n} \sum_{j=k+1}^{n}\widetilde\psi(r_{k,j})\cdot \Big (\log \frac{s_j-s_{k-1}}{s_j-s_k} - \log \frac{s_{j+1}-s_{k-1}}{s_{j+1}-s_k}\Big )\\
&\quad\quad\quad\quad\quad\quad\quad\quad\quad\quad \quad\quad \quad\quad\quad\quad\quad\quad\quad\quad -\sum_{k=1}^{n} \widetilde\psi (r_{k,k})\cdot \log \frac{s_{k+1}-s_{k-1}}{s_{k+1}-s_k},
  \end{align*} 
$$
\widetilde B= \widetilde\psi(r_{0,n+1})\cdot \log \frac{1}{s_{n+1}-s_0} +\sum_{j=1}^{n}\widetilde\psi(r_{0,j}) \cdot \log \frac{s_{j+1}-s_{0}}{s_{j}-s_0} -\widetilde\psi (r_{0,0})\cdot \log \frac{1}{s_{1}-s_0},
 $$
 $$
 \widetilde C=-\sum_{j=1}^n\widetilde\psi(r_{j,j})\cdot \log \frac{1}{s_j-s_{j-1}}.
 $$
 First notice that $r_{j,j}\in\{0,1\}$ for $j=1,\cdots,n$, thus $ \widetilde C=0$. Let $\widehat \psi(r)=(1-\delta) (r-1)$ for $r\ge 1$ and $\widehat \psi(0)=0$. We have $\widetilde \psi(r)\le \widehat \psi(r)$ for $1\le r\le \zeta-q$, and $\widetilde \psi(n+q\gamma)=n-1+q'=\widehat \psi(n+q')+\delta (n+q'-1)$ if $q<1$, as well as $\widetilde \psi(n+q)=n+q-1=\widehat \psi(n+q)+\delta (n+q-1)$ if $q=1$. Now, define formally $\widehat A$ and $\widehat B$ as $\widetilde A$ and $\widetilde B$, by replacing $\widetilde \psi$ by $\widehat \psi$. Notice that all the $\log \frac{1}{s_j-s_k}$ and $\Big (\log \frac{s_j-s_{k-1}}{s_j-s_k} - \log \frac{s_{j+1}-s_{k-1}}{s_{j+1}-s_k}\Big )$ are positive. Then, remembering that $r_{0,n+1}=n+q\gamma_q$ and rewriting $\widetilde \psi (r_{0,n+1})=\delta (r_{0,n+1}'-1)+\widehat  \psi(r_{0,n+1}')$ in expression $\widetilde B$, where $r_{0,n+1}'=n+q'$ if $q<1$ and $r_{0,n+1}'=n+q$ if $q=1$, and remembering also that $\widetilde\psi (r_{j,j})=\widehat  \psi(r_{j,j})$ for $j=0,\cdots,n$ since $r_{j,j}\in\{0,1\}$, the previous inequalities between $\widetilde \psi$ and $\widehat \psi$ yield:
 $$
 \sum_{k=0}^{n}\sum_{j=k+1}^{n+1} \alpha(j,k)\cdot \log \frac{1}{s_j-s_k}\le \delta (r_{0,n+1}'-1)\cdot \log \frac{1}{s_n-s_0}+\widehat A+\widehat B.
 $$
Now define $\widehat\beta(j,k):=\widehat\psi(r_{k,j})-\widehat\psi(r_{k,j-1})$. It is easy to see that $\widehat\beta(j,k)\le 1-\delta$ for $0\le k<j\le n+1$ since $r_{k,j}-r_{k,j-1} \le 1$ (when $q<1$, we have chosen $\gamma$ small enough such that $q\gamma<1$). Thus
 \begin{align*}
\widehat A&=\sum_{j=1}^{n+1}\sum_{k=0}^{j-1}\widehat\beta(j,k)\cdot \log \frac{s_j-s_{k-1}}{s_j-s_k}\le (1-\delta)\sum_{j=1}^n \log \frac{s_j-s_{0}}{s_j-s_{j-1}}\\ 
\widehat B&=\sum_{j=1}^{n+1}\widehat \beta (j,0)\cdot \log \frac{1}{s_j-s_0}\le (1-\delta)\sum_{j=1}^n\cdot \log \frac{1}{s_j-s_0}.
 \end{align*}
This gives
 $$
 \widehat A+\widehat B\le (1-\delta)  \sum_{j=1}^n\log \frac{1}{s_j-s_{j-1}},
 $$
and bounding $r_{0,n+1}-1$ by $n$ (we have chosen $q'<1$), we get 
$$
\sum_{k=0}^{n}\sum_{j=k+1}^{n+1} \alpha(j,k)\cdot \log \frac{1}{s_j-s_k}\le n\delta \cdot \log \frac{1}{s_{n+1}-s_0}+ (1-\delta)  \sum_{j=1}^{n+1}\log \frac{1}{s_j-s_{j-1}}.
$$
One has
\[
\begin{aligned}
&\quad\int_{0}^{1/2}\int_{1/2<s_2<\cdots <s_{n+1}<1} \frac{ds_{n+1}ds_{n}\cdots ds_2ds_0}{{\scriptstyle (s_{n+1}-s_0)^{n\delta}[(s_{n+1}-s_{n})\cdots(s_2-1/2) (1/2-s_0)]^{1-\delta}}}\\
&=\int_{0}^{1/2}\int_{1/2<s_2<\cdots <s_{n}<1}\int_0^{1-s_n} \frac{duds_{n}\cdots ds_2ds_0}{{\scriptstyle(u+s_{n}-s_0)^{n\delta}[u(s_n-s_{n-1})\cdots(s_2-1/2) (1/2-s_0)]^{1-\delta}}}\\
&=\frac{1}{\delta}\int_{0}^{1/2}\int_{1/2<s_2<\cdots <s_{n}<1}\int_0^{(1-s_n)^\delta} \frac{dvds_{n}\cdots ds_2ds_0}{{\scriptstyle (v^{1/\delta}+s_{n}-s_0)^{n\delta}[(s_n-s_{n-1})\cdots(s_2-1/2) (1/2-s_0)]^{1-\delta}}}\\
&\le\frac{2^{n/\delta}}{\delta}\int_{0}^{1/2}\int_{1/2<s_2<\cdots <s_{n}<1}\int_0^{(1-s_n)^\delta} \frac{dvds_{n}\cdots ds_2ds_0}{{\scriptstyle (v+(s_{n}-s_0)^\delta)^{n}[(s_n-s_{n-1})\cdots(s_2-1/2) (1/2-s_0)]^{1-\delta} }}\\
&\le\frac{2^{n/\delta}}{(n-1)\delta}\int_{0}^{1/2}\int_{1/2<s_2<\cdots <s_{n}<1} \frac{ds_{n}\cdots ds_2ds_0}{{\scriptstyle (s_{n}-s_0)^{(n-1)\delta}[(s_n-s_{n-1})\cdots(s_2-1/2) (1/2-s_0)]^{1-\delta}}}\\
&\le \cdots \cdots \\
&\le\frac{2^{(n+\cdots+2)/\delta}}{(n-1)!\delta}\int_{0}^{1/2}\int_{1/2}^1 \frac{ds_2ds_0}{(s_{2}-s_0)^{\delta}[(s_2-1/2) (1/2-s_0)]^{1-\delta}}\\
&\le\frac{2^{(n+\cdots+2+1)/\delta}}{(n-1)!}\int_{0}^{1/2} \log \frac{2}{1/2-s_0}\cdot \frac{ds_0}{(1/2-s_0)^{(1-\delta)}}\\
&< \infty.
\end{aligned}
\]
This yields $\mathbb{E}(\mu(I_0)\mu(I_1)^\kappa)<\infty$. \qed

\section{Proof of Theorem \ref{nm}}

The proof follows the same lines as that given in \cite{BaMa02} for compound Poisson cascades, and uses computations similar to those performed in \cite{RhVa11} to find the sufficient condition of the finiteness. 

\medskip

Let $J=[t_0,t_1]\in \mathcal{I}$. For $t\in J$ and $\epsilon <|J|$ we have
\[
V_\epsilon^J(t)=\widetilde{V}^J_\epsilon(t) \cup V^{J,l}(t) \cup V^{J,r}(t),
\]
where $\widetilde{V}^J_\epsilon(t)=V_\epsilon^J(t) \setminus V^J_{|J|}(t)$ and recall in Section \ref{pmmon} that
\begin{eqnarray*}
V^{J,l}(t) &=&\left\{z=x+iy\in V(t): |J| \le y < 2(t_1-x)\right\},\\
V^{J,r}(t) &=&\left\{z=x+iy\in V(t): |J| \le y \le 2(x-t_0)\right\}.
\end{eqnarray*}
Let $s\in \{l,r\}$. Recall in Lemma \ref{Ccq} that for $q\in I_\nu$ there exists a constant $C_q<\infty$ such that
\begin{equation}\label{bcq}
\mathbb{E}\left(\sup_{t\in J} e^{q\Lambda(V^{J,s}(t))}\right)\le C_q,
\end{equation}
and for $q\in \mathbb{R}$ there exists a constant $c_q>0$ such that
\begin{equation}\label{cq}
\mathbb{E}\left(\inf_{t\in J}e^{q\Lambda(V^{J,s}(t))}\right)\ge c_q.
\end{equation}
Let $\widetilde{\mu}_\epsilon^J(t)=Q(\widetilde{V}^J_\epsilon(t))\, dt$, $\widetilde{\mu}^J=\lim_{\epsilon \to0} \widetilde{\mu}^J_\epsilon$ and $\widetilde{Z}(J)=\widetilde{\mu}^J(J)/|J|$. Then it is easy to see that for $q\in I_\nu$,
\[
\mathbb{E}(\widetilde{Z}(J)^q)<\infty \Rightarrow \mathbb{E}(Z(J)^q)<\infty.
\]
and for $q\in\mathbb{R}$,
\[
\mathbb{E}(Z(J)^q)<\infty \Rightarrow \mathbb{E}(\widetilde{Z}(J)^q)<\infty.
\]

\subsection{} First we show that for $q\in I_\nu \cap (-\infty,0)$ we have $\mathbb{E}(Z^q)<\infty$. Let $J_0=I_{00}$ and $J_1=I_{11}$. It is clear that
\[
\widetilde{\mu}^I(I)\ge \widetilde{\mu}^I(J_0)+\widetilde{\mu}^I(J_1).
\]
For $i\in\{0,1\}$ define
\begin{eqnarray*}
V_i&=&V^I(J_i)\cap \{z\in \mathbb{H}: \mathrm{Im}(z) \le |I|\},\\
V_{i,l}(t)&=&V^{J_i,l}(t)\cap \{z\in \mathbb{H}: \mathrm{Im}(z) \le |I|\},\\
V_{i,r}(t)&=&V^{J_i,r}(t)\cap \{z\in \mathbb{H}: \mathrm{Im}(z) \le |I|\},
\end{eqnarray*}
and
\[
m_{i,l}=\inf_{t\in J_i} e^{\Lambda(V_{i,l}(t))}; \ m_{i,r}=\inf_{t\in J_i} e^{\Lambda(V_{i,r}(t))}.
\]

For $i=0,1$ let $U_i=4^{-1} \cdot m_{i,l} \cdot m_{i,r} \cdot e^{\Lambda(V_i)}$. Then we have
\[
\widetilde{Z}(I) \ge U_0 \widetilde{Z}(J_0)+U_1\widetilde{Z}(J_1),
\]
where $\widetilde{Z}(I)$, $\widetilde{Z}(J_0)$, $\widetilde{Z}(J_1)$ have the same law; $U_0$, $U_1$ have the same law; $\widetilde{Z}(J_0)$, $\widetilde{Z}(J_1)$ and $(U_0,U_1)$ are independent. So by using the approach of Molchan for Mandelbrot cascades in the general case \cite[Theorem 4]{Mol96}, we only need to show that $\mathbb{E}(U_0^q)<\infty$ to imply that $\mathbb{E}(\widetilde{Z}(I)^q)<\infty$, thus $\mathbb{E}(Z^q)<\infty$.

Since $q<0$, we have
\[
U_0^q=4^{-q}\cdot \sup_{t\in J_0} e^{q\Lambda(V_{0,l}(t))} \cdot \sup_{t\in J_0} e^{q\Lambda V_{0,r}(t))} \cdot  e^{q\Lambda(V_0)}.
\]
Notice that these random variables are independent, so
\[
\mathbb{E}(U_0^q)=4^{-q}\cdot \mathbb{E}\left(\sup_{t\in J_0}e^{q\Lambda(V_{0,l}(t))}\right) \cdot \mathbb{E}\left(\sup_{t\in J_0} e^{q\Lambda V_{0,r}(t))} \right)\cdot  \mathbb{E}\left(e^{q\Lambda(V_0)}\right).
\]
Then from the fact that $q\in I_\nu$ and \eqref{bcq} we get the conclusion. \qed

\subsection{} Now we show that for $q\in (-\infty,0)$, if $\mathbb{E}(Z^q)<\infty$ then $q\in I_\nu$. Let $J_0=\inf I +|I|[0,2/3]$, $J_1=\inf I+ |I|[1/3,1]$ and $J=\inf I +|I|[1/3,2/3]$. Then we have
\[
\widetilde{\mu}^I(I)\le \widetilde{\mu}^I(J_0)+\widetilde{\mu}^I(J_1).
\]
For $i\in\{0,1\}$ define
\begin{eqnarray*}
V_i&=&(V^I(J_i)\setminus V^I(J))\cap \{z\in \mathbb{H}: \mathrm{Im}(z) < |I|\},\\
V_{i,l}(t)&=&V^{J_i,l}(t)\cap \{z\in \mathbb{H}: \mathrm{Im}(z) < |I|\},\\
V_{i,r}(t)&=&V^{J_i,r}(t)\cap \{z\in \mathbb{H}: \mathrm{Im}(z) < |I|\}.
\end{eqnarray*}
Also define $V=V^I(J)\cap \{z\in \mathbb{H}: \mathrm{Im}(z) < |I|\}$. Then we get
\[
\widetilde{Z}(I) \le e^{\Lambda(V)}\cdot \left(\sum_{i=0,1} 4^{-1}\cdot \sup_{t\in J_i} e^{\Lambda(V_{i,l}(t))}\cdot \sup_{t\in J_i} e^{\Lambda(V_{i,l}(t))} \cdot e^{\Lambda(V_i)}\cdot \widetilde{Z}(J_i) \right).
\]
Since $q<0$, this gives
\[
\widetilde{Z}(I)^q \ge e^{q\Lambda(V)}\cdot \left(\sum_{i=0,1} 4^{-q}\cdot \inf_{t\in J_i} e^{q\Lambda(V_{i,l}(t))}\cdot \inf_{t\in J_i} e^{q\Lambda(V_{i,l}(t))} \cdot e^{q\Lambda(V_i)} \cdot \widetilde{Z}(J_i)^q \right).
\]
Taking expectation from both side and using \eqref{cq} we get
\[
\mathbb{E}(\widetilde{Z}(I)^q) \ge \mathbb{E}(e^{q\Lambda(V)})\cdot 2\cdot 4^{-q}\cdot c_q^2 \cdot \mathbb{E}(e^{q\Lambda(V_0)})\cdot \mathbb{E}(\widetilde{Z}(I)^q).
\]
Then from $\mathbb{E}(\widetilde{Z}(I)^q)<\infty$ we get $\mathbb{E}(e^{q\Lambda(V\cup V_0)})\le 2^{-1} 4^q c_q^{-2}<\infty$. This yields $q\in I_\nu$. \qed

\section{Proof of Theorem \ref{support}}

The proof is similar to that of \cite[Theorem 2.4]{Liu00}.

For $i\in \Sigma_*$ and $j\in\{0,1\}$ let $W_j^{[i]}=W_{ij}/ W_i$.

For $n\ge 1$, $\omega \in \Omega$ and $\mathbf{i}\in\Sigma$ define
\begin{eqnarray*}
A_n(\omega,\mathbf{i}) &=& \sum_{i=i_1\cdots i_n\in \Sigma_{n}} W_{i_n}^{[i_1\cdots i_{n-1}]}(\omega)\cdot \mathbf{1}_{\{\mathbf{i}|_n=i\}}\\
R_n(\omega,\mathbf{i}) &=& \sum_{i\in\Sigma_n} Z_i(\omega)\cdot \mathbf{1}_{\{\mathbf{i}|_n=i\}}.
\end{eqnarray*}
Thus for any $i=i_1\cdots i_n$ and $\mathbf{i}\in[i]$ we have 
$$
\mu(I_i)=\Big (\prod_{k=1}^nA_k(\omega,\mathbf{i})\Big ) \cdot R_n(\omega,\mathbf{i}).
$$ 
We claim that for any $n\ge 1$, $A_n$ has the same law as $A$, and  $R_n$ has the same law as $R$, where $A$ and $R$ are defined as in the beginning of Section~\ref{heart of the proof}; moreover,  $A_1,\cdots, A_n, R_n$ are independent. This is due to the fact that for any non-negative Borel functions $f_1,\cdots,f_n$ and $g$ one  gets
\begin{eqnarray*}
&& \mathbb{E}_\mathbb{Q}\left(g(R_n)\prod_{j=1}^kf_j(A_j)\right) \\
&=& \mathbb{E}\left(\sum_{i=i_1\cdots i_{n}\in\Sigma_n}  g(Z_{i}) Z_{i} \prod_{k=1}^{n}  f_k(W_{i_k}^{[i_1\cdots i_{k-1}]})W^{[i_1\cdots i_{k-1}]}_{i_{k}} \right)\\
&=&\mathbb{E}(g(Z)Z) \prod_{k=1}^n 2\mathbb{E}(f_k(W_0)W_0) \\
&=&\mathbb{E}_\mathbb{Q}(g(R)) \prod_{k=1}^n \mathbb{E}_\mathbb{Q}(f_k(A)).
\end{eqnarray*}

Under the assumptions we have
\[
\mathbb{E}_\mathbb{Q}(\log A)=2\mathbb{E}(W_0\log W_0)=\varphi'(1)\log 2:=\beta\in(-\infty,0)
\]
and
\[
\mathbb{E}_\mathbb{Q}((\log A)^2)-\mathbb{E}_\mathbb{Q}(\log A)^2=\varphi''(1)\log 2:=\gamma\in(0,\infty).
\]
Denote by $S_n=\log A_1+\cdots \log A_n$. By using law of iterated logarithm we get
\[
\limsup_{n\to\infty}  \frac{S_n-n\beta}{\sqrt{2\gamma n\log\log n}}=1, \ \mathbb{Q}\text{-a.s.}
\]
It follows that for $\mathbb{Q}$-almost all $(\omega,\mathbf{i})\in\Omega\times \Sigma$ and all $0<\epsilon <1$,
\begin{equation}\label{Sn}
e^{n\beta+(1-\epsilon)\sqrt{2\gamma n\log\log n}} \le e^{S_n} \le e^{n\beta+(1+\epsilon)\sqrt{2\gamma n\log\log n}},
\end{equation}
where the left inequality holds for infinitely many $n\in\mathbb{N}$, while the right inequality holds for all $n\in\mathbb{N}$ sufficiently large. We also have the following lemma.

\begin{lemma}\label{Zeps}
For $0<\epsilon<1$ one has for $\mathbb{Q}$-almost all $(\omega,\mathbf{i})\in\Omega\times \Sigma$ and all $n\in\mathbb{N}$ sufficiently large,
\[
e^{-\sqrt{n}\epsilon}\le R_n \le e^{\sqrt{n}\epsilon}. 
\]
\end{lemma}

Then the rest of the proof is exactly the same as \cite[Theorem 2.4]{Liu00}. \qed

\subsection{Proof of Lemma \ref{Zeps}}

The proof is borrowed from Lemma 12 in \cite{LiRo94}. First we have
\begin{eqnarray*}
\mathbb{Q}\left(|\log R_n| \ge \sqrt{n}\epsilon \right)&=&\mathbb{Q}\left( R_n \ge e^{\sqrt{n}\epsilon} \right)+\mathbb{Q}\left(R_n \le e^{-\sqrt{n}\epsilon} \right)\\
&=&\mathbb{E}\left( Z\cdot \mathbf{1}_{\{Z \ge e^{\sqrt{n}\epsilon}\}} \right)+\mathbb{E}\left(Z\cdot \mathbf{1}_{\{Z \le e^{-\sqrt{n}\epsilon}\}} \right)\\
&\le&\mathbb{E}\left( Z\cdot \mathbf{1}_{\{Z \ge e^{\sqrt{n}\epsilon}\}} \right)+e^{-\sqrt{n}\epsilon}.
\end{eqnarray*}
Applying the elementary inequality $\sum_{n\ge 1} \mathbf{1}_{\{ X\ge \sqrt{n}\}} \le X^2$ we get
\begin{eqnarray*}
\sum_{n\ge 1}\mathbb{Q}\left(|\log R_n| \ge \sqrt{n}\epsilon \right) &\le&\sum_{n\ge 1}\mathbb{E}\left( Z\cdot \mathbf{1}_{\{Z \ge e^{\sqrt{n}\epsilon}\}} \right)+\sum_{n\ge 1}e^{-\sqrt{n}\epsilon}\\
&=&\mathbb{E}\left( Z\cdot \sum_{n\ge 1} \mathbf{1}_{\left\{\frac{\log Z}{\epsilon} \ge \sqrt{n}\right\}} \right)+\sum_{n\ge 1}e^{-\sqrt{n}\epsilon}\\
&\le &\epsilon^{-2}\mathbb{E}(Z(\log Z)^2)+\sum_{n\ge 1}e^{-n\epsilon}.
\end{eqnarray*}
Since $\varphi'(1)<0$, there exists $q>1$ such that $\varphi(q)<0$, thus due to Theorem \ref{pm} we have $\mathbb{E}(Z^q)<\infty$. This implies $\mathbb{E}(Z (\log Z)^2)<\infty$, and the conclusion comes from Borel-Cantelli lemma.

\bibliographystyle{habbrv}
\bibliography{refs}

\end{document}